\documentclass[11pt]{amsart}
\usepackage{geometry}                % See geometry.pdf to learn the layout options. There are lots.
\usepackage{amsfonts, amscd, amssymb, amsthm, amsmath, stmaryrd,shuffle}
\usepackage[nice]{nicefrac}
\usepackage{pdfsync}%leaves makers for tex searching
\usepackage{enumerate}
\usepackage[svgnames]{xcolor}

\usepackage{tikz}
\usepgflibrary{shapes.geometric}
\tikzstyle{V}=[draw, fill =black, circle, inner sep=0pt, minimum size=1.5pt]
\newcommand\TikZ[1]{\begin{matrix}\begin{tikzpicture}#1\end{tikzpicture}\end{matrix}}

\newcounter{r}
\newcommand\Part[1]{
        \setcounter{r}{1}
	 \foreach \x in {#1}{
 	{\ifnum\value{r}=1
		\draw (0,\value{r}-1)--(\x,\value{r}-1); 
		\fi}
	\draw (0,\value{r}) to (\x,\value{r});
   	\foreach \y in {0, ..., \x} {\draw (\y,\value{r})--(\y,\value{r}-1);}
	\addtocounter{r}{1}
 }}

\theoremstyle{plain}
	\newtheorem{thm}{Theorem}[section]
	\newtheorem{lemma}[thm]{Lemma}
	\newtheorem{prop}[thm]{Proposition}
	\newtheorem{cor}[thm]{Corollary}
\theoremstyle{definition}
	\newtheorem{defn}[thm]{Definition}
	
    \newtheorem{notn}[thm]{Notation}
	\newtheorem*{note}{Note}
    \newtheorem*{akn}{Acknowledgement}
	\newtheorem{example}{Example}
\def\NR{Notation\ }
\usepackage{color}
\definecolor{dred}{rgb}{.65, 0, 0.15}

\newcommand{\splt}[3]{[\operatorname{split}_{{#1}}({#2})]_{{#3}}}
\newcommand{\spl}[3]{\operatorname{split}_{{#1}}({#2})}

\def\cA{\mathcal{A}}\def\cO{\mathcal{O}}\def\cR{\mathcal{R}}\def\cS{\mathcal{S}}

  \def\CC{\mathbb{C}}                       \def\ZZ{\mathbb{Z}}

\def\cleq{\preccurlyeq}
\def\cgeq{\succcurlyeq}

\DeclareMathOperator{\Sym}{Sym}
\DeclareMathOperator{\QS}{QSym}
\DeclareMathOperator{\NS}{NSym}
\DeclareMathOperator{\set}{Set}
\DeclareMathOperator{\comp}{comp}
\DeclareMathOperator{\lp}{lp}

\DeclareMathOperator{\spab}{sp}
\DeclareMathOperator{\OSP}{OSP}
\DeclareMathOperator{\word}{\pi}
\def\h{\boldsymbol{h}}
\def\e{\boldsymbol{e}}
\def\y{\boldsymbol{y}}

\def\p{\boldsymbol{\Psi}}
\def\ptwo{\boldsymbol{\Phi}}
\def\r{\boldsymbol{r}}
\def\cO{\mathcal{O}}
\def\<{\langle}
\def\>{\rangle}
\def\C{\mathbb{C}}
\def\Cons{\mathrm{Cons}} 
\def\Sh{\mathrm{Sh}} 
\def\Br{\mathrm{Br}}
\def\cS{A}

\title{Quasisymmetric Power Sums}
\author{Cristina Ballantine, Zajj Daugherty, Angela Hicks, Sarah Mason, Elizabeth Niese}
\date{\today}                                           % Activate to display a given date or no date

\DeclareSymbolFont{bbold}{U}{bbold}{m}{n}
\DeclareSymbolFontAlphabet{\mathbbold}{bbold}
\def\one{\mathbbold{1}}

\begin{document}
\begin{abstract} In the 1995 paper entitled ``Noncommutative symmetric functions," Gelfand, et.\ al.\ defined two noncommutative symmetric function analogues for the power sum basis of the symmetric functions, along with analogues for the elementary and the homogeneous bases.  They did not consider the noncommutative symmetric power sum duals in the quasisymmetric functions, which have since been explored only in passing by Derksen and Malvenuto-Reutenauer.  These two distinct quasisymmetric power sum bases are the topic of this paper.  In contrast to the simplicity of the symmetric power sums, or the other well known bases of the quasisymmetric functions, the quasisymmetric power sums have a more complex combinatorial description.  As a result, although symmetric function proofs often translate directly to quasisymmetric analogues, this is not the case for quasisymmetric power sums. Neither is there a model for working with the quasisymmetric power sums in  the work of Gelfand, et.\ al., which relies heavily on quasi-determinants (which can only be exploited by duality for our purposes) and is not particularly combinatorial in nature. This paper therefore offers a first glimpse at working with these two relatively unstudied quasisymmetric bases, avoiding duality where possible to encourage a previously unexplored combinatorial understanding. \end{abstract}
\maketitle
\setcounter{tocdepth}{1}
\tableofcontents

\section{Introduction}

The ring of symmetric functions $\Sym$ has several well-studied bases indexed by integer partitions  $\lambda$, such as the monomial basis $m_\lambda$, the elementary  basis $e_\lambda$, the complete homogeneous basis $h_\lambda$, the Schur functions $s_\lambda$, and, most relevant here, the power sum basis $p_\lambda$. Two important generalizations of $\Sym$ are $\QS$ (the ring of quasisymmetric functions) and $\NS$ (the ring of noncommutative symmetric functions). These rings share dual Hopf algebra structures, giving a rich interconnected theory with many beautiful algebraic and combinatorial results.  In particular, many quasisymmetric and noncommutative symmetric analogues  to the familiar symmetric bases  have been defined and studied, such as the quasisymmetric monomial basis  $M_\alpha$, and the noncommutative elementary and homogeneous bases $\e_\alpha$ and  $\h_\alpha$ \cite{GKLLRT94} (where the indexing set is compositions $\alpha$). Several different analogues of the Schur functions  have also been defined, including the quasisymmetric fundamental basis $F_\alpha$ \cite{gessel1984multipartite}, dual to the noncommutative ribbon basis $\r_\alpha$; the quasi-Schur basis and its dual in \cite{haglund2011quasisymmetric}; and the immaculate basis and its quasisymmetric dual~\cite{BBSSZ14lift}.  

Quasisymmetric analogues of symmetric function bases are useful for a number of reasons.  Quasisymmetric functions form a combinatorial Hopf algebra~\cite{Ehr96,gessel1984multipartite,MalReu95} and in fact are the terminal object in the category of combinatorial Hopf algebras~\cite{ABS06}, which explains why they consistently appear throughout algebraic combinatorics.  Complicated combinatorial objects often have simpler formulas when expanded into quasisymmetric functions, and translating from symmetric to quasisymmetric functions can provide new avenues for proofs.

Here, we explore the analogs to the power sum bases. In $\Sym$, there is an important bilinear pairing, the Hall inner product, defined by $\<m_\lambda, h_\mu\> = \delta_{\lambda, \mu}$. Moreover, the duality between $\QS$ and $\NS$ precisely generalizes the inner product on $\Sym$ so that, for example,  $\<M_\lambda, \h_\mu\> = \delta_{\lambda, \mu}$. With respect to the pairing on $\Sym$, the power sum basis is (up to a constant) self-dual, so analogs to the power sum basis in $\QS$ and $\NS$ should share a similar relationship. Two types of noncommutative power sum bases, $\p_\alpha$ and $\ptwo_\alpha$, were defined by Gelfand, et.\ al.\ \cite{GKLLRT94}. Briefly, the quasisymmetric duals to one type or the other were also discussed in \cite{der09} and in \cite{MalReu95}; but  in contrast to the other bases listed above, very little has been said about their structure or their relationship to other bases. The main objective of this paper is to fill this gap in the literature.  Namely, we define two types of quasisymmetric power sum bases, which are scaled duals to $\p_\alpha$ and $\ptwo_\alpha$. The scalars are chosen analogous to the scaled self-duality of the symmetric power sums; moreover, we show that these are exactly the right coefficients to force our bases to refine the symmetric power sums (Theorems~\ref{thm:refine} and~\ref{thm:2refine}). Section~\ref{sec:qsps} develops combinatorial proofs of these refinements. In Section~\ref{sec:btw}, we give transition matrices to other well-understood bases. Section~\ref{sec:products} explores algebraic properties, giving explicit formulas for products of quasisymmetric power sums. Section~\ref{sec:plethysm} gives formulas for plethysm in the quasisymmetric case.

\section{Preliminaries}\label{sec:prelim}
In this section, we define the rings $\QS$ of quasisymmetric functions and $\NS$ of noncommutative symmetric functions, and briefly discuss their dual Hopf algebra structures. 

We begin with a brief discussion of notation.  Due to the nature of this paper, we note that there in a lot of notation to keep track of throughout, and therefore we set aside numbered definitions and notations to help the reader. In general, we use lower case letters (e.g.\ $e, m, h, s$, and $p$) to indicate {\em symmetric functions}, bold lowercase letters (e.g.\ $\e$, $\h$, and $\r$) to indicate {\em noncommutative symmetric functions}, and capital letters (e.g.\ $M$ and $F$) to indicate {\em quasisymmetric functions}.  When there is a single clear analogue of a symmetric function basis, we use the same letter for the symmetric functions and their analogue (following \cite{LMvW} rather than \cite{GKLLRT94}).  For the two different analogs to the power sums, we echo \cite{GKLLRT94} in using $\p$ and $\ptwo$ for the noncommutative symmetric power sums, and then $\Psi$ and $\Phi$ as quasisymmetric analogues. We generally follow \cite{LMvW} for the names of the automorphisms on the quasisymmetric and noncommutative symmetric functions.  For example, we use $S$ for the antipode map (in particular, see \cite[\S 3.6]{LMvW} for a complete list and a translation to other authors).

\subsection{Quasisymmetric functions}\label{sec:qsym}
A formal power series $f \in \C\llbracket x_1,x_2,\ldots \rrbracket$ is a {\em quasisymmetric function} if the coefficient of $x_1^{a_1}x_2^{a_2}\cdots x_k^{a_k}$ in $f$ is the same as the coefficient for $x_{i_1}^{a_1}x_{i_2}^{a_2}\cdots x_{i_k}^{a_k}$ for any $i_1<i_2<\cdots <i_k$. The set of quasisymmetric functions $\QS$ forms a ring. Moreover, this ring has a $\ZZ_{\geq 0}$-grading by degree, so that $\QS=\bigoplus_n\QS_n$, where $\QS_n$ is the set of $f \in \QS$ that are homogeneous of degree $n$. For a comprehensive discussion of $\QS$ see \cite{LMvW,MalReu95,Sta99v2}.  

There are a number of common bases for $\QS_n$ as a vector space over $\C$.  These bases are indexed by (strong) integer compositions.  
\begin{defn}[composition, $\alpha\vDash n$]
A sequence $\alpha=(\alpha_1,\alpha_2,\ldots,\alpha_k)$ is a \emph{composition} of $n$, denoted $\alpha\vDash n$, if $\alpha_i>0$ for each $i$ and $\sum_i \alpha_i=n$. \end{defn} 

\begin{notn}[$|\alpha|$, $l(\alpha)$, $\widetilde{\alpha}$]The {\em size} of a composition $\alpha=(\alpha_1,\alpha_2,\ldots,\alpha_k)$ is $|\alpha|=\sum \alpha_i$ and the {\em length} of $\alpha$ is $\ell(\alpha)=k$.  Given a composition $\alpha$, we denote by  $\widetilde{\alpha}$  the partition obtained by placing the parts of $\alpha$ in weakly decreasing order.
\end{notn}

\begin{defn}[refinement, $\beta\cleq\alpha$, $\beta^{(j)}$]\label{defn:refinement}
If $\alpha$ and $\beta$ are both compositions of $n$, we say that $\beta$ {\em refines} $\alpha$ (equivalently, $\alpha$ is a {\em coarsening} of $\beta$), denoted $\beta\cleq \alpha$, if $$\alpha=(\beta_1+\cdots+\beta_{i_1}, \beta_{i_1+1}+\cdots +\beta_{i_1+i_2}, \ldots, \beta_{i_1+\cdots +i_{k-1}+1}+\cdots + \beta_{i_1+\cdots+i_k}).$$  
We will denote by $\beta^{(j)}$ the composition made up of the parts of $\beta$ (in order) that sum to $\alpha_{j}$; namely, if $j=i_s$, then $\beta^{(j)}=(\beta_{i_1+\cdots +i_{s-1}+1},\cdots , \beta_{i_1+\cdots+i_s})$.
\end{defn}
It is worth noting that some authors reverse the inequality, using $\cleq$ for coarsening as opposed to refinement as we do here.  Repeatedly we will need the particular parts of $\beta$ that sum to a particular part of $\alpha$.

\begin{notn}[$\set(\alpha)$, $\comp(A)$] There is a natural bijection between compositions of $n$ and subsets of $[n-1]$ given by partial sums. (Here $[n]$ is the set $\{1,2, \hdots , n\}$.)  Namely, if $\alpha=(\alpha_1,\ldots,\alpha_k)\vDash n$, then $\set(\alpha) = \{\alpha_1, \alpha_1+\alpha_2, \ldots, \alpha_1+\cdots+\alpha_{k-1}\}.$  Similarly, if $A=\{a_1,\ldots,a_j\}\subseteq[n-1]$ with $a_1<a_2<\cdots<a_j$ then $\comp(A)=(a_1,a_2-a_1,\ldots, a_j-a_{j-1},n-a_j)$.
\end{notn} 
We remark that  $\alpha \cleq \beta$ if and only if $\set(\beta)\subseteq \set(\alpha)$.

Let $\alpha=(\alpha_1,\ldots,\alpha_k)$ be a composition. The {\em quasisymmetric monomial function} indexed by $\alpha$ is 
\begin{equation}\label{eq:Ms}M_\alpha=\sum_{i_1<i_2<\cdots<i_k} x_{i_1}^{\alpha_1}x_{i_2}^{\alpha_2}\cdots x_{i_k}^{\alpha_k};\end{equation}
and the {\em fundamental quasisymmetric function} indexed by $\alpha$ is 
\begin{equation} \label{M-F} F_\alpha = \sum_{\beta\cleq \alpha}M_\beta, \qquad \text{so that} \qquad M_\alpha = \sum_{\beta\cleq\alpha}(-1)^{\ell(\beta)-\ell(\alpha)}F_\beta.\end{equation}
Equivalently, $F_{\alpha}$ is defined directly by
\begin{equation}\label{eq:Fs}
	F_{\alpha} = \sum_{\substack{i_1\leq i_2\leq \cdots \leq i_n\\  i_j<i_{j+1} \text{ if } j \in \set(\alpha)}} x_{i_1}x_{i_2}\cdots x_{i_n}.\end{equation}

In addition to being a graded ring, $\QS$ can be endowed with the structure of a \emph{combinatorial Hopf algebra}.  For our purposes, this means that $\QS$ has a product (ordinary polynomial multiplication), a coproduct $\Delta$, a unit and counit, and an antipode map.  The ring $\NS$ of noncommutative symmetric functions is dual to $\QS$ with respect to a certain inner product (defined later), and thus also is a combinatorial Hopf algebra. For further details on the Hopf algebra structure of $\QS$ and $\NS$, see~\cite{ABS06,grinberg2014hopf,LMvW}.

\subsection{Noncommutative symmetric functions}\label{sec:nsym}  The ring of {\em noncommutative symmetric functions}, denoted $\NS$, is formally defined as a free associative algebra $\C \langle \e_1, \e_2, \hdots \rangle$, where the $\e_i$ are regarded as {\em noncommutative elementary functions} and 
	\[
		\e_\alpha = \e_{\alpha_1}\e_{\alpha_2}\cdots \e_{\alpha_k}, \qquad 
		\text{for a composition } \alpha.\]  
Define the {\em noncommutative complete homogeneous symmetric functions} as in~\cite[\S 4.1]{GKLLRT94} by 
\begin{equation}\label{eq:hine}
\h_n = \sum_{\alpha \vDash n} (-1)^{n-\ell(\alpha)}\e_\alpha, \quad \text{ and } \quad 
\h_\alpha = \h_{\alpha_1}\cdots \h_{\alpha_k} = \sum_{\beta\cleq \alpha}(-1)^{|\alpha|-\ell(\beta)}\e_\beta.\end{equation}
The noncommutative symmetric analogue (dual) to the fundamental quasisymmetric functions is given by the {\em ribbon Schur functions} 
\begin{equation}\label{eq:rinh}
\r_\alpha = \sum_{\beta\cgeq \alpha} (-1)^{\ell(\alpha)-\ell(\beta)}\h_\beta.\end{equation}
\subsubsection{Noncommutative power sums}
To define the noncommutative power sums, we  begin by recalling the useful exposition in~\cite[\S 2]{GKLLRT94} on the (commuting) symmetric power sums.  Namely, the power sums $p_n$ can be defined by the generating function:
$$P(X;t)=\sum_{k\geq 1}t^{k-1}p_k[X]=\sum_{i\geq 1}x_i(1-x_it)^{-1}.$$

This generating function can equivalently be defined by any of the following generating functions, where $H(X;t)$ is the standard generating function for the complete homogeneous functions and $E(X;t)$ is the standard generating function for the elementary homogeneous functions:
\begin{equation}
P(X;t)=\frac{d}{dt}\log H(X;t) = -\frac{d}{dt}\log E(-X;t).\label{eq:pfrome}
\end{equation}
Unfortunately, there is not a unique sense of logarithmic differentiation for power series (in $t$) with noncommuting coefficients (in $\NS$). Two natural well-defined reformulations of these are 
\begin{equation}\label{eq:type1gen}
\frac{d}{dt}H(X;t)=H(X;t)P(X;t) \quad \text{ or } \quad
-\frac{d}{dt}E(X;-t)=P(X;t)E(X;-t),
\end{equation}
and 
\begin{equation}
H(X;t)=-E(X;-t) = \exp\left(\int P(X;t) dt\right). \label{eq:type2gen}
\end{equation}
In $\NS$, these do indeed give rise to \emph{two different analogs} to the power sum basis, introduced in~\cite[\S 3]{GKLLRT94}: the \emph{noncommutative power sums of the first kind} (or \emph{type}) $\p_\alpha$ and of the \emph{second kind} (or \emph{type}) $\ptwo_\alpha$, with explicit formulas (due to \cite[\S 4]{GKLLRT94}) as follows.

The noncommutative power sums of the first kind are those satisfying essentially the same generating function relation as \eqref{eq:type1gen}, where this time $H(X;t)$, $E(X;t)$, and $P(X;t)$ are taken to be the generating functions for the noncommutative homogeneous, elementary, and type one power sums respectively, and expand as
\begin{equation}\label{eq:powerinh}
\p_n = \sum_{\beta \vDash n} (-1)^{\ell(\beta)-1} \beta_k \h_\beta 
\end{equation}
where $\beta=(\beta_1,\ldots,\beta_k)$. 
\begin{notn}[$\lp(\beta,\alpha)$]Given a composition $\alpha =(\alpha_1, \ldots, \alpha_m)$ and a composition $\beta = (\beta_1,\ldots, \beta_k)$ which refines $\alpha$, we let $\lp(\beta) = \beta_k$ (last part) and  $$ \lp(\beta,\alpha) = \prod_{i=1}^{\ell(\alpha)} \lp(\beta^{(i)}).$$ \end{notn}  
Then
\begin{equation}\label{eq:htopsi}\p_\alpha =\p_{\alpha_1}\cdots\p_{\alpha_m}=\sum_{\beta \cleq \alpha} (-1)^{\ell(\beta)-\ell(\alpha)}\lp(\beta,\alpha)\h_\beta.
\end{equation}
Similarly, the noncommutative power sums of the second kind are those satisfying the analogous generating function relation to \eqref{eq:type2gen}, and expand as
\begin{equation}\label{eq:htophi}
\ptwo_n = \sum_{\alpha\vDash n}(-1)^{\ell(\alpha)-1}\frac{n}{\ell(\alpha)}\h_\alpha, \quad  \text{and} \quad  \ptwo_\alpha = \sum_{\beta\cleq\alpha}(-1)^{\ell(\beta)-\ell(\alpha)}\frac{\prod_i \alpha_i}{\ell(\beta,\alpha)}\h_\beta,
\end{equation} 
where $\ell(\beta,\alpha)=\prod_{j=1}^{\ell(\alpha)} \ell(\beta^{(j)})$. 

\subsection{Dual bases}\label{sec:dual}Let $V$ be a vector space over $\C$, and let $V^* = \{ \text{linear } \varphi: V \to \CC\}$ be its dual. Let $\<,\>: V\otimes V^* \rightarrow \C$ be the natural bilinear pairing. Bases of these vector spaces are indexed by a common set, say $I$; and we say bases $\{b_\alpha\}_{\alpha \in I}$ of $V$ and $\{b_\alpha^*\}_{\alpha \in I}$ of $B^*$ are \emph{dual} if $\<b_\alpha,b^*_\beta\>=\delta_{\alpha,\beta}$ for all $\alpha, \beta \in I$.  

Due to the duality between $\QS$ and $\NS$, we make extensive use of the well-known relationships between change of bases in a vector space an its dual. Namely, if $(A,A^*)$ and $(B,B^*)$ are two pairs of dual bases of $V$ and $V^*$, then 
for $a_{\alpha}\in A$ and $b_{\beta}^* \in B^*$, we have 
\[a_\alpha = \sum_{b_\beta \in B}c_\beta^\alpha b_\beta\qquad \text{ if and only if }\qquad b^*_\beta = \sum_{a^*_\alpha \in A^*}c_\beta^\alpha a^*_\alpha.\]

\noindent In particular, the bases $\{M_\alpha\}$ of $\QS$ and $\{\h_\alpha\}$  of $\NS$ are dual; as are  $\{F_\alpha\}$ and $\{\r_\alpha\}$ (see \cite[\S 6]{GKLLRT94}). The primary object of this paper is to explore properties of two $\QS$ bases dual to $\{\p_\alpha\}$  or $\{ \ptwo_\alpha \}$ (up to scalars) that also refine $p_\lambda$.  Malvenuto and Reutenauer~\cite{MalReu95} mention (a rescaled version of) the type 1 version but do not explore its properties; Derksen \cite{der09} describes such a basis for the type 2 version, but a computational error leads to an incorrect formula in terms of the monomial quasisymmetric function expansion. 

 \section{Quasisymmetric power sum bases}\label{sec:qsps}
The symmetric power sums have the property that $\<p_\lambda,p_\mu\> = z_\lambda \delta_{\lambda, \mu}$ where $z_\lambda$ is as follows. 
\begin{notn}[$z_\alpha$]\label{notn:z} 
	For a  partition $\lambda \vdash n$, let $m_i$ be the number of parts of length $i$. Then 
$$z_\lambda = 1^{m_1}m_1!2^{m_2}m_2!\cdots k^{m_k}m_k!.$$ 
Namely, $z_\lambda $ is the size of the stabilizer of a permutation of cycle type $\lambda$ under the action of $S_n$ on itself by conjugation. For a composition $\alpha$, we use $z_{\alpha}=z_{\widetilde{\alpha}}$, where $\widetilde{\alpha}$ is the partition rearrangement of $\alpha$ as above.
\end{notn} 
 We describe two quasisymmetric analogues of the power sums, each of which satisfies a variant of this duality property.

\subsection{Type 1 quasisymmetric power sums}
We define the type 1 quasisymmetric power sums to be the basis $\Psi_\alpha$ of $\QS$ such that 
$$\<\Psi_\alpha,\p_\beta\> = z_\alpha \delta_{\alpha,\beta}.$$ 
While duality makes most of this definition obvious, the scaling is somehow a free choice to be made. However, as we show in Theorem \ref{thm:refine} and Corollary \ref{cor:refine}, our choice of scalar not only generalizes the self-dual relationship of the symmetric power sums, but serves to provide a refinement of those power sums. Moreover, the proof leads to a (best possible) combinatorial interpretation of $\Psi_\alpha$.

In \cite[\S 4.5]{GKLLRT94}, the authors  give both the transition matrix from the $\h$ basis to the $\Psi$ basis (above in (\ref{eq:powerinh})), and is inverse. Using the latter and duality, we   compute a monomial expansion of $\Psi_{\alpha}$.

\begin{notn}[$\pi(\alpha,\beta)$]
First, given $\alpha$ a refinement of $\beta$, recall from Definition \ref{defn:refinement} that $\alpha^{(i)}$ is the composition consisting of the parts of $\alpha$ that combine (in order) to $\beta_i$. Define 
	$$\pi(\alpha)=\prod_{i=1}^{\ell(\alpha)} \sum_{j=1}^i \alpha_j
		\qquad\text{and} \qquad 
		\pi(\alpha,\beta)=\prod_{i=1}^{\ell(\beta)} \pi(\alpha^{(i)}).$$ 
\end{notn}

Then
\[\h_\alpha = \sum_{\beta \preccurlyeq \alpha} \frac{1}{\pi(\beta,\alpha)}\p_\beta.\] 
By duality, the polynomial 
\[\psi_\alpha = \sum_{\beta\cgeq \alpha}\frac{1}{\pi(\alpha,\beta)}M_\beta\]
has the property that $\<\psi_\alpha,\p_\beta\>=\delta_{\alpha,\beta}$.  Then the type 1 quasisymmetric power sums have the following monomial expansion:
\begin{equation}\label{eq:PsiM}\Psi_\alpha = z_\alpha \psi_\alpha=z_\alpha\sum_{\beta\cgeq\alpha}\frac{1}{\pi(\alpha,\beta)}M_\beta.\end{equation}

For example \begin{align*}
\Psi_{232} &= (2^2 \cdot 2! \cdot 3)(\frac{1}{2 \cdot 3 \cdot 2} M_{232}+\frac{1}{2 \cdot 5 \cdot 2} M_{52} + \frac{1}{2 \cdot 3 \cdot 5} M_{25} + \frac{1}{2 \cdot 5 \cdot 7} M_7) \\
&= 2 M_{232} +\frac{6}{5} M_{52} + \frac{4}{5} M_{25}+\frac{12}{35} M_7.
\end{align*}

The remainder of this section is devoted to constructing  the ``best possible'' combinatorial formulation of the $\Psi_\alpha$, given in Theorem \ref{thm:combPsi}, followed by the proof of the refinement of the symmetric power sums, given in Theorem \ref{thm:refine} and Corollary \ref{cor:refine}. 

\bigskip

\subsubsection{A combinatorial interpretation of $\Psi_\alpha$}
We consider the set $S_n$ of permutations of $[n]=\{1,2,\dots,n\}$ both in one-line notation and in cycle notation. For a partition $\lambda=(\lambda_1, \lambda_2, \ldots, \lambda_{\ell})$ of $n$, a permutation $\sigma$ has \emph{cycle type} $\lambda$ if its cycles are of lengths $\lambda_1$, $\lambda_2$, \dots, $\lambda_\ell$.  We consider two canonical forms for writing a permutation according to its cycle type.

\begin{defn}[standard  and partition forms]\label{def:standard-partition}
 A permutation in cycle notation is said to be in \emph{standard form} if each cycle is written with the largest element last and the cycles are listed in increasing order according to their largest element.  It is said to be in \emph{partition form} if each cycle is written with the largest element last, the cycles are listed in descending length order, and cycles of equal length are listed in increasing order according to their largest element.
\end{defn}
For example, for the permutation $(26)(397)(54)(1)(8)$, we have standard form $(1)(45)(26)(8)(739)$ and partition form $(739)(45)(26)(1)(8)$.  Note that our definition of standard form differs from that in \cite[\S1.3]{Sta99v1} in the cyclic ordering; they are equivalent, but our convention is more convenient for the purposes of this paper.  If we fix an order of the cycles (as we do when writing a permutation in standard and partition forms), the \emph{(ordered) cycle type} is the composition $\alpha \vDash n$ where the $i$th cycle has length $\alpha_i$.  

As alluded to in Notation \ref{notn:z}, we have  \cite[Prop.\ 1.3.2]{Sta99v1}
\begin{equation*} 
 \frac{n!}{z_\lambda} = \#\{ \sigma \in S_n \text{ of cycle type } \lambda \}.
\end{equation*}

We are now ready to define a subset of $S_n$ (which uses one-line notation) needed to prove Proposition~\ref{prop:consistent}.
\begin{notn}[$\splt{\beta}{\sigma}{j}$] Let $\beta \vdash n$ and let $\sigma \in S_n$ be written in one-line notation. First partition $\sigma$ according to $\beta$ (which we draw using $|\!|$), and consider the (disjoint) words $\spl{\beta}{\sigma}==[\sigma^{1}, \dots, \sigma^{\ell}]$, where $\ell = \ell(\beta)$.  Let $\splt{\beta}{\sigma}{j}=\sigma^j$. See Table \ref{tbl:consistent}.
\end{notn}
\begin{defn}[consistent, $\Cons_{\alpha \cleq \beta}$]\label{defn:consistent}
Fix  $\alpha \cleq \beta$ compositions of $n$. Given  $\sigma \in S_n$ written in one-line notation,  let $\sigma^j=\splt{\beta}{\sigma}{j}$.  Then, for each $i = 1, \dots, \ell$, add parentheses to $\sigma^{i}$ according to $\alpha^{(i)}$, yielding disjoint permutations  $\bar{\sigma}^{i}$ (of subalphabets of $[n]$) of cycle type $\alpha^{(i)}$. If the resulting subpermutations $\bar{\sigma}^{i}$ are all in standard form, we say $\sigma$ is \emph{consistent} with $\alpha \cleq \beta$. In other words, we look at subsequences of $\sigma$ and split according to $\beta$ \textit{separately} to see if, for each $j$, the $j$th subsequence is in standard form when further partition by $\alpha^{(j)}$. Define 
\begin{equation*}
\Cons_{\alpha \cleq \beta} = \{\sigma \in S_n \mid \sigma \mbox{ is consistent with } \alpha \cleq \beta\}.
\end{equation*}

\end{defn}

\begin{example}\label{ex:consistent}
Fix $\alpha = (1,1,2,1,3,1)$ and $\beta = (2,2,5)$.  Table~\ref{tbl:consistent} shows several examples of permutations and the partitioning process.\end{example}
\begin{table}[h]
$$
\begin{array}{c@{\qquad}c@{\qquad}c@{\qquad}c}
\text{permutation $\sigma$} & \text{partition by $\beta$} & \text{add $()$ by $\alpha$} & \text{$\sigma$ consistent?}\\\hline
571423689 	
	& \underbrace{57}_{{\sigma}^{1}}|\!|\underbrace{14}_{{\sigma}^{2}}|\!|\underbrace{23689}_{{\sigma}^{3}}	
	& \underbrace{(5)(7)}_{\bar{\sigma}^{1}}|\!|\underbrace{(14)}_{\bar{\sigma}^{2}}|\!|\underbrace{(2)(368)(9)}_{\bar{\sigma}^{3}}	&\text{yes}\\
571428369 	
	& \underbrace{57}_{{\sigma}^{1}}|\!|\underbrace{14}_{{\sigma}^{2}}|\!|\underbrace{28369}_{{\sigma}^{3}}		
	& \underbrace{(5)(7)}_{\bar{\sigma}^{1}}|\!|\underbrace{(14)}_{\bar{\sigma}^{2}}|\!|\underbrace{(2)(\boldsymbol{8}36)(9)}_{\bar{\sigma}^{3}}
	&\textbf{no}\\
571493682 	
	& \underbrace{57}_{{\sigma}^{1}}|\!|\underbrace{14}_{{\sigma}^{2}}|\!|\underbrace{93682}_{{\sigma}^{3}}		
	& \underbrace{(5)(7)}_{\bar{\sigma}^{1}}|\!|\underbrace{(14)}_{\bar{\sigma}^{2}}|\!|\underbrace{(\boldsymbol{9})(36\boldsymbol{8})(\boldsymbol{2})}_{\bar{\sigma}^{3}}		
	&\textbf{no}
\end{array}
$$
\caption{Examples of permutations in $S_9$ and determining if they are in $\Cons_{\alpha\cleq\beta}$ where $\alpha=(1,1,2,1,3,1)$ and $\beta=(2,2,5)$.  Note how $\beta$ subtly influences consistency in the last example.}\label{tbl:consistent}
\end{table}

We also consider the set of all permutations consistent with a given $\alpha$ and all possible choices of (a coarser composition) $\beta$, as each will correspond to various monomial terms in the expansion of a given $\Psi_\alpha$.
\begin{example} We now consider sets of permutations that are consistent with $\alpha=(1,2,1)$ and each coarsening of $\alpha$. The coarsening of $\alpha = (1,2,1)$ are $(1,2,1)$, $(3,1)$, $(1,3)$, and $(4)$, and the corresponding sets of consistent permutations in $S_4$ are \begin{align*}
\Cons_{(1,2,1) \cleq (1,2,1)}& = \{1234, 1243, 1342, 2134, 2143, 2341, 3124, 3142, 3241, 4123, 4132, 4231\},
\\\Cons_{(1,2,1) \cleq (1,3)} &=  \{1234, 2134, 3124, 4123\},\\
\Cons_{(1,2,1) \cleq (3,1)} &= \{1234, 1243, 1342, 2134, 2143, 2341, 3142, 3241\},
\\\Cons_{(1,2,1) \cleq (4)}& = \{1234, 2134\}. \end{align*} Notice that these sets are not disjoint.  
\end{example}
The following is a salient observation that can be seen from this example.
\begin{lemma}\label{lemma:niceobs}If $\sigma$ is consistent with $\alpha\cleq \beta$ for some choice of $\beta$, then $\sigma$ is consistent with $\alpha \cleq\gamma$ for all $\alpha \cleq \gamma \cleq \beta$.\end{lemma}
Note that this implies $Cons_{\alpha \cleq \beta} \subseteq Cons_{\alpha \cleq \alpha}$ for all $\alpha \cleq \beta$.  With these examples in mind, we will use the following lemma to justify a combinatorial interpretation of $\Psi_\alpha$ in Theorem~\ref{thm:combPsi}.
\begin{lemma}\label{lem:cons} If $\alpha \cleq \beta$, we have $$ n!=|\Cons_{\alpha \cleq \beta}|\cdot \pi(\alpha, \beta).$$ 
\end{lemma}
\begin{proof}Consider the set
\[\cS_\alpha= \bigotimes_{i=1}^{\ell(\beta)}\left( \bigotimes_{j=1}^{\ell(\alpha^{(i)})} \nicefrac{\mathbb{Z}}{a_j^{(i)}\mathbb{Z}}\right), \qquad \text{where } a_j^{(i)} = \sum_{r=1}^j \alpha_r^{(i)}.\]
We have
 \[|\cS_\alpha| = \prod_{i=1}^{\ell(\beta)} \prod_{j=1}^{\ell(\alpha^{(i)})}a_j^{(i)} = \pi(\alpha, \beta).\]
Construct a map
\begin{align*}
 \Sh: \Cons_{\alpha\cleq\beta} \times  \cS_\alpha &\to S_n\\
 	(\sigma, s) \quad&\mapsto \sigma_s
\end{align*}
 as follows (see also Example \ref{ex:Rlambdabeta-identity(b)}).

For $s=[s^{(i)}_j]_{i=1\ j=1}^{\ell(\beta)\ \ell(\alpha^{(i)})} \in \cS_\alpha$ and $\sigma \in \Cons_{\alpha\cleq\beta}$, construct a permutation $\sigma_s \in S_n$ as follows. 
\begin{enumerate}[1.]
\item Partition $\sigma$  into  words $\sigma^{1}, \dots, \sigma^{\ell}$ according to $\beta$ so that $\sigma^{i}=\splt{\beta}{\sigma}{i}$. 
\item For each $i=1, \ldots, \ell(\beta)$, modify $\sigma^{i}$ % new text follows
by cycling the first $a_j^{(i)}$ values right by $s_j^{(i)}$ for $j=1, \ldots,\ell(\alpha^{(i)})$.
Call the resulting word $\sigma^{i}_s$.
\item Let $\sigma_s = \sigma^{1}_s \cdots \sigma^{\ell}_s$. 
\end{enumerate}

 \noindent This process is invertible as follows. Let $\tau \in S_n$ be written in one-line notation. 
\begin{enumerate}[1'.]
\item Partition $\tau$  into  words $\tau^{1}, \dots, \tau^{\ell}$ according to $\beta$ such that $\tau^i=\splt{\beta}{\tau}{i}$.
\item For each $i=1, \ldots, \ell(\beta)$, let $m_i = \ell(\alpha^{(i)})$. Modify $\tau^{i}$ and record $s^{(i)}_j$ %new text follows
for $j=m_i,\ldots,1$ by cycling the first $a_j^{(i)}$ values left until the largest element is last.  Let $s_j^{(i)}$ be the number of required shifts left.
Call the resulting word $\sigma^{i}$.
\item Let $\sigma= \sigma^{1} \cdots \sigma^{\ell}$ and $s  =[s^{(i)}_j]_{i=1\ j=1}^{\ell(\beta)\ \ell(\alpha^{(i)})}$. 
\end{enumerate}
By construction,  $s \in \cS_\alpha$ and $\sigma \in \Cons_{\alpha \cleq \beta}$. It is straightforward to verify that $\sigma_s = \tau$. Therefore $\Sh^{-1}$ is well-defined, so that $\Sh$ is a bijection, and thus $n!=|\Cons_{\alpha\cleq\beta}|\cdot \pi(\alpha,\beta)$.
\end{proof}
\begin{example}\label{ex:Rlambdabeta-identity(b)}
 As an example of the construction of $\Sh$ in the proof of Lemma~\ref{lem:cons}, let $\beta=(5,4) \vDash 9$, and let  $\alpha=(2,3,2,2) \cleq \beta$, so that 
 $$\alpha^{(1)} = (2,3), \quad a_1^{(1)} = 2, \quad a_2^{(1)} = 2+3=5, \quad \text{ and }$$
 $$\alpha^{(2)} = (2,2), \quad a_1^{(2)} = 2, \quad a_2^{(2)} = 2+2=4. \phantom{\quad \text{ and }}$$
 Fix  $\sigma=267394518 \in \Cons_{\alpha \cleq \beta}$, and $s=(s^{(1)},s^{(2)})=((1,3),(0,1)) \in \cS_\alpha$. We want to determine $\sigma_s$. 
 
 \begin{enumerate}[1.]
 \item Partition $\sigma$ according to $\beta$: \quad 
 $\sigma^{1} = 26739$ and $\sigma^{2} = 4518.$
 
 \item Cycle $\sigma^{i}$ according to $\alpha^{(i)}$:
 
 \begin{align*}
 \sigma^{1} = 26739  \xrightarrow{\text{take first $a_1^{(1)} = 2$ terms}}  	&\ \underline{26}739
 		\xrightarrow{\text{cycle $s_1^{(1)} = 1$ right}}  					\ \underline{62}739\\
 		\xrightarrow{\text{take first $a_2^{(1)} = 5$ terms}}  					&\ \underline{62739}
 		\xrightarrow{\text{cycle $s_2^{(1)} = 3$ right}}  					\ \underline{73962} =  \sigma^{1}_s;
\end{align*}

\begin{align*}
 \sigma^{2} = 4518  \xrightarrow{\text{take first $a_1^{(2)} = 2$ terms}}  	&\ \underline{45}18
 		\xrightarrow{\text{cycle $s_1^{(2)} = 0$ right}}  					\ \underline{45}18\\
 		\xrightarrow{\text{take first $a_2^{(2)} = 4$ terms}}  					&\ \underline{4518}
 		\xrightarrow{\text{cycle $s_2^{(2)} = 1$ right}}  					\ \underline{8451} =  \sigma^{2}_s.
\end{align*}
 
 \item Combine to get $\sigma_s = \sigma^{1}_s \sigma^{2}_s = 739628451$.
 \end{enumerate}

 \noindent Going the other way, start with $\beta=(5,4)\vDash 9$,  $\alpha=(2,3,2,2) \cleq \beta$, and $\tau = 739628451 \in S_9$ in one-line notation, and want to find $\sigma \in \Cons_{\alpha \cleq \beta}$ and $s \in S$ such that $\sigma_s = \tau$. 
\begin{enumerate}[1'.]
\item Partition $\tau$ according to $\beta$: \quad 
 $\tau^{1} = 73962$ and $\tau^{2} = 8451.$

\item Cycle $\tau^{i}$ into  $\alpha \cleq \beta$ consistency and record shifts:
{\setlength{\fboxsep}{1.5pt}
 \begin{align*}
 \tau^{1} = 73962 \xrightarrow{\text{take first $a_2^{(1)} = 5$ terms}} &\ \stackrel{(\xleftarrow{~3~}) }{\underline{73\fbox{$9$}62}}
  \xrightarrow{\text{cycle left so largest is last, and record}}   \underline{6273\fbox{$9$}},  & s_2^{(1)} = 3,\\
   \xrightarrow{\text{take first $a_1^{(1)} = 2$ terms}} &\ \stackrel{(\xleftarrow{1})}{\underline{\fbox{$6$}2}}\!\!739 
  \xrightarrow{\text{cycle left so largest is last, and record}}  \underline{2\fbox{$6$}}739 = \sigma^{1},  & s_1^{(1)} = 1;\\
\end{align*}

 \begin{align*}
 \tau^{2} = 8451 \xrightarrow{\text{take first $a_2^{(2)} = 4$ terms}} & \stackrel{(\xleftarrow{~1~})}{\underline{\fbox{$8$}451}}
  \xrightarrow{\text{cycle left so largest is last, and record}}  \underline{451\fbox{$8$}},  & s_2^{(2)} = 1,\\
   \xrightarrow{\text{take first $a_1^{(2)} = 2$ terms}} &\ \stackrel{\checkmark}{\underline{4\fbox{5}}}18
  \xrightarrow{\text{cycle left so largest is last, and record}}  \underline{4\fbox{$5$}}18 = \sigma^{2},  & s_1^{(2)} = 0.\\
\end{align*}}

\item Combine to get $\sigma = \sigma^{1} \sigma^{2} = 267394518$ and $s = ((1,3),(0,1))$ as expected.

\end{enumerate}
\end{example}

One might reasonably ask for a ``combinatorial'' description of the quasisymmetric power sums, one similar to our description of the monomial and fundamental basis of the quasisymmetric functions as in (\ref{eq:Ms}) and (\ref{eq:Fs}). There is not an altogether satisfactory such formula for either of the quasisymmetric power sums, although the previous lemma hints at what appears to be the best possible interpretation in the Type 1 case.  We give the formula, and its quick proof.  Before we begin, we will find the following notation useful both here and in the fundamental expansion of the Type 1 power sums.

\begin{notn}[$\widehat{\alpha(\sigma)}$, $\Cons_{\alpha}$]\label{notn:hat} Given a permutation $\sigma$ and a composition $\alpha$, let $\widehat{\alpha(\sigma)}$ denote the coarsest composition $\beta$ with $\beta \cgeq \alpha$ and $\sigma \in \Cons_{\alpha\cleq\beta}.$ For example, if $\alpha=(3,2,2)$ and $\sigma=1352467$, then $\widehat{\alpha(\sigma)}=(3,4)$.  In addition, we write $\sigma \in \Cons_\alpha$ if we are considering $\beta=\alpha$.\end{notn}% \NOTE{(this $\hat$ the same as taking the intersection of the partitions in each of compositions, thinking of them in stars and bars notation) and $\alpha^c$=???.}

\begin{thm}\label{thm:combPsi}  
Let $m_i(\alpha)$ the multiplicity of $i$ in $\alpha$.  Then
	$$\Psi_{(\alpha_1,\cdots,\alpha_k)}(x_1,\cdots,x_m)
			=\frac{ \prod_{i=1}^n m_i(\alpha)!}{n!}\sum_{\sigma\in S_n}
				\sum_{\substack{1\leq i_1\leq \cdots\leq i_k\leq m\\
				i_j=i_{j+1}\Rightarrow\\ 
				\max(\splt{\alpha}{\sigma}{j})<\max(\splt{\alpha}{\sigma}{j+1})}}x_{i_1}^{\alpha_1}\cdots x_{i_k}^{\alpha_k},$$
                where $\max(\splt{\alpha}{\sigma}{j})$ is the maximum of the list defined in Notation 3.4.
\end{thm}
\begin{proof}
First, by Lemmas \ref{lem:cons} and \ref{lemma:niceobs},
 \begin{align}
\Psi_\alpha &=\frac{z_\alpha}{n!}\sum_{\alpha\cleq\beta}|\Cons_{\alpha\cleq\beta}|M_\beta \nonumber\\
&=\frac{z_\alpha}{n!}\sum_{\sigma\in \Cons_\alpha}\sum_{\alpha\cleq\delta\cleq \widehat{\alpha(\sigma)}}M_\delta
\label{eq:unscaled}\\&= \frac{ \prod_{i=1}^n m_i(\alpha)!}{n!}\sum_{\sigma\in \Cons_\alpha}\left(\prod_{i}i^{m_i(\alpha)}\right)\sum_{\alpha\cleq\delta\cleq \widehat{\alpha(\sigma)}}M_\delta.\nonumber\end{align}
The first equality follows by grouping together terms according to the permutation counted rather than the monomial basis.  Next, for each $\sigma\in \Cons_\alpha$, we can assign $\left(\prod_{i}i^{m_i(\alpha)}\right)$ objects by cycling each $\sigma$ (within cycles defined by $\alpha$) in all possible ways.  (Thus for all $j$ we cycle $\splt{\alpha}{\sigma}{j}$.)  The result for a fixed $\alpha$ is all permutations of $\mathfrak{S}_n$ (considered in one line notation by removing the cycle marks). In particular, for any new permutation $\tau$ we may recover the original permutation $\sigma$ by cycling each $\splt{\alpha}{\tau}{j}$, until the maximal term in each cycle is again at the end.  Thus we may instead sum over all permutations and consider the largest element (rather than the last element) in each cycle:  
\begin{align*}
\Psi_\alpha 
&=\frac{ \prod_{i=1}^n m_i(\alpha)!}{n!}
	\sum_{\sigma\in S_n}
	\sum_{\substack{1\leq i_1\leq \cdots\leq i_k\leq m\\
					i_j=i_{j+1}\Rightarrow\\ 
					\max(\splt{\alpha}{\sigma}{j})<\max(\splt{\alpha}{\sigma}{j+1})}}
			x_{i_1}^{\alpha_1}\cdots x_{i_k}^{\alpha_k}.
\end{align*}
 
\end{proof}
While one might hope to incorporate the multiplicities (perhaps summing over a different combinatorial object, or considering cycling parts, then sorting them by largest last element) there does not seem to be a natural way to do so with previously well known combinatorial objects; the heart of the problem is that the definition of consistency inherently uses (sub)permutations written in standard form, while $\frac{n!}{z_\alpha}$ counts permutations with cycle type $\alpha$ in partition form.  This subtlety blocks simplification of formulas throughout the paper. In practice, we expect (\ref{eq:unscaled}) to be a more useful expression because of this fact, but work out the details of the other interpretation here as it cleanly expresses this easily overlooked subtlety.  

\subsubsection{Type 1 quasisymmetric power sums refine symmetric power sums}
We next turn our attention to a proof that the type 1 quasisymmetric power sums refine the symmetric power sums in a natural way.
\begin{notn}[$R_{\alpha\beta}$, $\cO_{\alpha\beta}$]\label{notn:R}
For compositions $\alpha$, $\beta$, let 
\begin{equation*}
R_{\alpha\beta} = | \cO_{\alpha\beta}|, \text{ where } \cO_{\alpha\beta} = \left\{ \left. \begin{matrix}\text{ordered set partitions}\\\text{$(B_1,\cdots, B_{\ell(\beta)})$ of $\{1,\cdots,\ell(\alpha)\}$}\end{matrix} ~\right|  ~
		\beta_j=\sum_{i\in B_j}\alpha_i \text{ for } 1\leq j\leq \ell(\beta) \right\},
		\end{equation*}
i.e.\ $R_{\alpha\beta}$ is the number of ways to group the parts of $\alpha$ so that the parts in the $j$th (unordered) group sum to $\beta_j$.
\end{notn}
\begin{example} If $\alpha=(1,3,2,1)$ and $\beta=(3,4)$, then $\cO_{\alpha\beta}$ consists of 
$$(\{2\}, \{1,3,4\}),\  (\{1,3\}, \{2,4\}), \quad \text{and}  \quad (\{ 3,4\},\{1,2 \}),$$ 
 corresponding, respectively, to the three possible ways of grouping parts of $\alpha$,
$$(\alpha_2 , \alpha_1+\alpha_3+\alpha_4),\  (\alpha_1+\alpha_3 , \alpha_2+\alpha_4), \quad \text{and} \quad ( \alpha_3+\alpha_4 , \alpha_1+\alpha_2).$$ 
Therefore $R_{\alpha \beta} = 3$.
\end{example}

For partitions $\lambda, \mu$, we have 
\begin{equation*}\label{eq:p-in-m}
	p_\lambda=\sum_{\mu \vdash n}R_{\lambda \mu}m_\mu
\end{equation*}
(see for example \cite[p.297]{Sta99v2}).
Further, if $\widetilde{\alpha}$ is the partition obtained by putting the parts of $\alpha$ in decreasing order as before, then 
\begin{equation*}\label{eq:p-in-M}
R_{\alpha\beta}=R_{\widetilde{\alpha}\widetilde{\beta}}, \qquad \text{ implying }\qquad p_\lambda=\sum_{\alpha \models n}R_{\lambda \alpha}M_\alpha.
\end{equation*}

The refinement of the symmetric power sums can be established either by exploiting duality or through a bijective proof.  We present the bijective proof first (using it to justify our combinatorial interpretation of the basis) and defer the simpler duality argument to \S~\ref{sec:products}. 
\begin{thm}\label{thm:refine}
Let $\lambda \vdash n$.  Then 
\[p_\lambda = \sum_{\alpha: \widetilde{\alpha}=\lambda} \Psi_\alpha.\]
\end{thm}
\begin{cor}\label{cor:refine}  $\Psi_\alpha=z_\alpha\psi_\alpha=z_{\widetilde{\alpha}}\psi_\alpha$ is the unique rescaling of the $\psi$ basis (that is the dual basis to $\p$) which refines the symmetric power sums with unit coefficients.
\end{cor}

Recall from \eqref{eq:PsiM} that for a composition $\alpha$, 
$$\Psi_\alpha = \sum_{\beta \cgeq \alpha}\frac{z_\alpha}{\pi(\alpha,\beta)}M_\beta.$$  
Summing over $\alpha$ rearranging to $\lambda$ and multiplying on both sides by ${n!}/{z_\lambda}$, we see that to prove Theorem \ref{thm:refine} it is sufficient to establish the following for a fixed $\beta\vDash n$.

\begin{prop}\label{prop:consistent}
For $\lambda \vdash n$ and $\beta \vDash n$, 
\[R_{\lambda\beta}\frac{n!}{z_{\lambda}}=\sum_{\substack{\alpha \cleq \beta \\ \widetilde{\alpha}=\lambda}}\frac{n!}{\pi(\alpha, \beta)}.\]
\end{prop}
\begin{proof} The proof is in two steps, with the first being Lemma~\ref{lem:cons}.  Let $\beta\vDash n$ and $\lambda \vdash n$.  We must establish that \begin{equation}\label{R=cons}R_{\lambda\beta} \frac{n!}{z_{\lambda}} = \sum_{\substack{\alpha \cleq \beta\\\widetilde{\alpha}=\lambda}}|\Cons_{\alpha\cleq\beta}|.\end{equation}

Let $\cO_{\lambda\beta}$ be  the set of ordered set partitions as  defined in \NR \ref{notn:R}. For each refinement $\alpha \cleq \beta$, let  
$C_{\alpha}=\{(\alpha, \sigma) \mid \sigma \mbox{ is consistent with } \alpha \cleq \beta\}$ and define $ C=\bigcup_{\substack{\alpha \cleq \beta \\ \widetilde{\alpha}=\lambda}} C_{\alpha}.$
Denote by $S_n^{\lambda}$ the set of permutations of $n$ of cycle type $\lambda$. Then we prove \eqref{R=cons}, 
by defining the map 
	$$\Br: C \to \cO_{\lambda\beta} \times S_n^{\lambda}$$ 
as follows (see also Example \ref{ex:Rlambdabeta-identity(a)}), and showing that it is a bijection.  Start with $(\alpha, \sigma)\in C$, with $\sigma$ written in one-line notation. 
\begin{enumerate}[1.]
\item Add parentheses to $\sigma$ according to $\alpha$, and denote the corresponding permutation (now written in cycle notation) $\bar{\sigma}$.
\item Sort the cycles of $\bar{\sigma}$ into partition form (as in Definition \ref{def:standard-partition}), and let $c_i$ be the $i$th cycle in this ordering. 
\item Comparing to $\bar{\sigma}^{1}$, \dots, $\bar{\sigma}^{\ell}$ as in Definition \ref{defn:consistent} of consistent permutations, define $B=(B_1, \dots, B_k)$ by $j \in B_i$ when $c_j$ belongs to $\bar{\sigma}^{i}$,   i.e.\ $\bar{\sigma}^{i} = \prod_{j \in B_i} c_j$.
\end{enumerate} 
Define $\Br(\alpha, \sigma)=(B, \bar{\sigma}).$  Since $\alpha$ rearranges to $\lambda$, $\bar{\sigma}$ has (unordered) cycle type $\lambda$.  And since $\prod_{j \in B_i} c_j = \bar{\sigma}^{i},$ we have $\sum_{j \in B_i}\lambda_j=\beta_i.$  Thus 
$\Br: (\alpha, \sigma) \mapsto (B, \bar{\sigma})$ is well-defined.

Next, we show that $\Br$ is invertible, and therefore a bijection. Namely, fix $B = (B_1, \ldots, B_{\ell}) \in \cO_{\lambda \beta}$ and $\bar{\sigma} \in S_n^\lambda$, writing $\bar{\sigma} = c_1 c_2 \cdots c_k$ in partition form. Then determine $(\alpha, \sigma) \in C$ as follows. 
\begin{enumerate}[1'.]
\item Let $\bar{\sigma}^{i} = \prod_{j \in B_i} c_j$,% notation?
and sort $\bar{\sigma}^{i}$ into standard form (as a permutation of the corresponding subalphabet of $[n]$). 
\item Let $\alpha$ be the (ordered) cycle type of $\bar{\sigma}^{1}\bar{\sigma}^{2}\cdots\bar{\sigma}^{\ell}$. 
\item Delete the parentheses. Let $\sigma$ be the corresponding permutation written in one-line notation. 
\end{enumerate}
By construction, $\alpha$ refines $\beta$ and is a rearrangement of $\lambda$, and $\sigma$ is ($\alpha \cleq \beta$)-consistent. And it is straightforward to verify that this process exactly inverts $\Br$. Therefore $\Br^{-1}$ is well-defined.  This implies that $\Br$ is a bijection and hence \eqref{R=cons} holds. %$R_{\lambda\beta} \frac{n!}{z_{\lambda}} = \sum_{\substack{\alpha \cleq \beta\\\widetilde{\alpha}=\lambda}}|\Cons_{\alpha\cleq\beta}|.$

Then it follows from Lemma \ref{lem:cons} that 
\[R_{\lambda\beta}\frac{n!}{z_\lambda}=\sum_{\substack{\alpha\cleq\beta\\\widetilde{\alpha}=\lambda}}\frac{n!}{\pi(\alpha,\beta)}\]
as desired.
\end{proof}

\begin{example}\label{ex:Rlambdabeta-identity(a)} As an example of the construction of $\Br$ in the proof of Proposition~\ref{prop:consistent}, let $\beta=(5,4)$, $\alpha=(2,3;2,2)$ and $\sigma=267394518$. We want to determine $\Br(\alpha,\sigma)$.

\begin{enumerate}[1.]
\item Add parentheses to $\sigma$ according to $\alpha$: \ \ $\bar{\sigma} = (26)(739)(45)(18)$.
\item Partition-sort the cycles of $\bar{\sigma}$:  $\bar{\sigma} = \underbrace{(739)}_{c_1}\underbrace{(45)}_{c_2}\underbrace{(26)}_{c_3}\underbrace{(18)}_{c_4}$.
\item Compare to the $\beta$-partitioning: $\underbrace{(26)(739)}_{\bar{\sigma}^{1}}|\!|\underbrace{(45)(18)}_{\bar{\sigma}^{2}}$.
So $B=(\{1,3\},\{2,4\})$, since $\bar{\sigma}^{1} = c_1 c_3$ and $\bar{\sigma}^{2} = c_2 c_4$. 
\end{enumerate}

\noindent Going the other way, start with  $B=(\{1,3\},\{2,4\})$  and  $\bar{\sigma} = (739)(45)(26)(18)$  written in partition form. 
\begin{enumerate}[1'.]
\item Place cycles into groups according to $B$:  $(739)(26) |\!| (45)(18)$.
\item Sort within parts in ascending order by largest value: $(26) (739)|\!| (45)(18)$\\
	Then $\alpha = (2,3,2,2)$.
\item Delete parentheses to get $\sigma = 267394518$.
\end{enumerate}
\end{example}

\subsection{Type 2 quasisymmetric power sums}\label{sec:type2power}
In order to describe the second type of quasisymmetric power sums, we introduce the following notation.

\begin{notn}[$\spab(\alpha,\beta)$] In the following, $\spab(\beta,\alpha) = \prod_i \spab(\beta^{(i)})$ for $\spab(\gamma)=\ell(\gamma)!\prod_j\gamma_j$.  
\end{notn}

As shown in~\cite{GKLLRT94}, we can write the noncommutative complete homogeneous functions in terms of the noncommutative power sums of type 2 as 
\[\h_\alpha = \sum_{\beta\cleq\alpha}\frac{1}{\spab(\beta,\alpha)}\ptwo_\beta.\] 

By duality, the basis dual to $\ptwo_\beta$ can be written as a sum of monomial symmetric functions as %to define $\phi_\alpha$ as the basis dual to $\ptwo$ where 
\[\phi_\alpha = \sum_{\beta\cgeq\alpha}\frac{1}{\spab(\alpha,\beta)}M_\beta.\] %Derkson  first did this calculation in \cite{der09}, but his result contains a miscalculation in the final formula.

We then define the type 2 quasisymmetric power sums as
\[\Phi_\alpha = z_\alpha  \phi_\alpha.\]  

A similar polynomial $P_\alpha$ is defined in~\cite{MalReu95}, and is related to $\phi_\alpha$ by $\phi_\alpha=\left(\prod_i \alpha_i\right)^{-1} P_\alpha$.  Note that this means Malvenuto and Reutenaur's polynomial $P_\alpha$ is not dual to $\ptwo_\alpha$ and (by the following results) does not refine the symmetric power sums.  For example, $\Phi_{322} = 2M_{322}+M_{52}+M_{34}+\frac{1}{3}M_7$ whereas $P_{322}=M_{322}+\frac{1}{2}M_{52}+\frac{1}{2}M_{34}+\frac{1}{6}M_7$.  %Summing $\Phi_{\alpha}$ over all $\alpha$ which rearrange to $(3,2,2)$ produces $2m_{322}+m_{43}+2m_{52}+m_7=p_{322}$ whereas summing $P_{\alpha}$ over all $\alpha$ which rearrange to $(3,2,2)$ produces $m_{322}+m_{52}+\frac{1}{2}m_{43}+\frac{1}{2}m_7 \not= p_{322}$.

We can obtain a more combinatorial description for $\Phi_\alpha$ by rewriting the coefficients and interpreting them in terms of ordered set partitions. 
\begin{notn}[$\OSP(\alpha,\beta)$]\label{notn:OSP} Let $\alpha \cleq \beta$ and let $\OSP(\alpha,\beta)$ denote the ordered set partitions of $\{1,\ldots,\ell(\alpha)\}$ with block size $|B_i|=\ell(\alpha^{(i)})$. If $\alpha \not\cleq \beta$, we set $\OSP(\alpha,\beta)=\emptyset$.
\end{notn}
\begin{thm}\label{thm:power2}
Let $\alpha\vDash n$ and let $m_i$ denote the number of parts of $\alpha$ of size $i$.  Then
\[\Phi_\alpha = \binom{\ell(\alpha)}{m_1,m_2,\ldots,m_k}^{-1}\sum_{\beta\cgeq\alpha}|\OSP(\alpha,\beta)|M_\beta.\]
\end{thm}
\begin{proof}
Given $\alpha\vDash n$, let $m_i$ denote the number of parts of $\alpha$ of size $i$.  Then 
\begin{align*}
\Phi_\alpha&=\sum_{\beta\cgeq\alpha} \frac{z_\alpha}{\spab(\alpha,\beta)}M_\beta\\
&= \sum_{\beta\cgeq\alpha}\frac{z_\alpha}{\prod_j \alpha_j\prod_i (\ell(\alpha^{(i)}))! } M_\beta\\
&=\frac{z_\alpha}{\ell(\alpha)!\prod_j \alpha_j}\sum_{\beta\cgeq\alpha}\frac{\ell(\alpha)!}{\prod_i(\ell(\alpha^{(i)}))!}M_\beta\\
&=\binom{\ell(\alpha)}{m_1,m_2,\ldots,m_k}^{-1}\sum_{\beta\cgeq\alpha}\frac{\ell(\alpha)!}{\prod_i(\ell(\alpha^{(i)}))!}M_\beta.
\end{align*}
Note that $\dfrac{\ell(\alpha)!}{\prod_i(\ell(\alpha^{(i)}))!}$ is the number of ordered set partitions of $\{1,\ldots,\ell(\alpha)\}$ with block size $B_i=\ell(\alpha^{(i)})$. Thus 
\[\Phi_\alpha = \binom{\ell(\alpha)}{m_1,m_2,\ldots,m_k}^{-1}\sum_{\beta\cgeq\alpha}|\OSP(\alpha,\beta)|M_\beta.\qedhere\]
\end{proof}

\begin{thm}{\label{thm:2refine}}
The type 2 quasisymmetric power sums refine the symmetric power sums by 
$$p_\lambda = \sum_{\widetilde{\alpha}=\lambda}\Phi_\alpha.$$
\end{thm}

Here the proof requires only a single (and less complex) bijection. 
\begin{lemma}\label{lem:type2refine}
Let $\lambda\vdash n$ and $\beta \vDash n$.  Let $m_i$ denote the number of parts of $\lambda$ of size $i$.  Then 
\[\binom{\ell(\lambda)}{m_1,m_2,\ldots,m_k}R_{\lambda\beta} = \sum_{\substack{\alpha\cleq\beta\\\widetilde{\alpha}=\lambda}}|\OSP(\alpha,\beta)|.\]
\end{lemma}
\begin{proof}
Let $\lambda \vdash n$, and $m_i$ denote the number of parts of $\lambda$ of size $i$, so that $\ell(\lambda) = \sum_{i=1}^{\lambda_1} m_i$. 
We can model a composition $\alpha$ that rearranges $\lambda$ as an ordered set partition $(A_1,\ldots,A_{\lambda_1})$ of $\{1,\cdots,\ell(\lambda)\}$ where 
$A_i = \{j ~|~ \alpha_j = i\}$.
Thus, if 
$$\cA_\lambda 
	= \left\{\left. \begin{matrix}\text{ordered set partitions}\\\text{$(A_1,\ldots,A_{\lambda_1})$ of $\{1,\cdots,\ell(\lambda)\}$}\end{matrix} ~\right|  ~
		|A_i| = m_i \right\},
		$$
then the map
$$\gamma: \cA \to \{\alpha \vDash n ~|~ \tilde{\alpha} = \lambda\}$$ 
		defined by 
\begin{equation}
\gamma(A) \text{ is the composition with $\gamma(A)_j=i$ for all $j \in A_i$}
\label{eq:gamma(A)}
\end{equation}
is a natural bijection. 
Further, we have $|\mathcal{A}_\lambda|=\binom{\ell(\lambda)}{m_1,m_2,\ldots,m_{\lambda_1}}$.

Now, fix $\beta \vDash n$. Recall, we have 
$$\cO_{\lambda\beta}= \left\{\left. \begin{matrix}\text{ordered set partitions}\\\text{$(B_1,\cdots, B_{\ell(\beta)})$ of $\{1,\cdots,\ell(\lambda)\}$}\end{matrix} ~\right|  ~
		\beta_j=\sum_{i\in B_j}\lambda_i \text{ for } 1\leq j\leq \ell(\beta) \right\},
		$$
so that $|\cO_{\lambda\beta}| = \cR_{\lambda \beta}$ (Notation \ref{notn:R}).  For $\alpha\cleq\beta$, we have 
$$\OSP(\alpha, \beta) = \left\{\left. \begin{matrix}\text{ordered set partitions}\\\text{$(C_1,\cdots, C_{\ell(\beta)})$ of $\{1,\cdots,\ell(\alpha)\}$}\end{matrix} ~\right|  ~ |C_i| = \ell(\alpha^{(i)})\right\}$$
(Notation \ref{notn:OSP}). Informally, $\cO_{\lambda\beta}$ tells us how to build $\beta$ as combination of parts of $\lambda$; and $\OSP(\alpha, \beta)$ tells us a shuffle of a refinement $\alpha \cleq \beta$.

For an ordered set partition $P = (P_1, \dots, P_\ell)$ of $\{1, \dots, \ell(\lambda)\}$, let $p_1^{(i)}, \dots, p_{\ell(\alpha^{(i)})}^{(i)}$ be  the elements of $P_i$ written in increasing order. Define $w_{P}$ be the permutation (in one-line notation) given by
$$w_{P_i} = p_1^{(i)} \cdots p_{\ell}^{(i)}\qquad \text{ and } \qquad w_P = w_{p_1} \cdots w_{p_{\ell(\beta)}}.$$

We are now ready to construct a bijection 
$$g: \cA_\lambda \times \cO_{\lambda\beta} \to  \bigsqcup_{\substack{\alpha\cleq\beta\\\widetilde{\alpha}=\lambda}}\{(\alpha, C) ~|~ C \in \OSP(\alpha, \beta)\}.$$ 
See Example \ref{ex:OSP}.

Let $(A, B) \in \cA_\lambda \times \cO_{\lambda\beta}$. Initially, set $\alpha' = \gamma(A)$ (where $\gamma$ is the map in \eqref{eq:gamma(A)}), and set $C\in \cO_{\alpha'\beta}$ equal to the image of $B$ under the permutation of indices induced by $\lambda \to \gamma(A)$. Namely, if the $i$th part of $\lambda$ got placed into the $j$th part of $\alpha'$ (where parts of equal length are kept in the same relative order), then $i$ in $B$ is replaced by $j$ in $C$. Now, act by $w_C^{-1}$ on the subscripts of $\alpha'$ to get $\alpha$. The result is $\alpha \cleq \beta$, $\tilde{\alpha} = \lambda$, and $C \in \OSP(\alpha, \beta)$. Let $g((A,B)) = (\alpha, C)$. 

To see that this is a bijection, we show that each step in building $g((A,B))$ is invertible as follows. Take $\alpha \cleq \beta$ such that $\tilde{\alpha} = \lambda$, and let $C \in \OSP(\alpha, \beta)$. Let $\alpha'$ be the result of acting by $w_C$ on the subscripts of $\alpha$. Then we can recover $A = \gamma^{-1}(\alpha')$ from $\alpha'$; and $B$ is the image of $C$ under the permutation of indices induced by $\alpha' \to \lambda$. Namely, if the $j$th part of $\alpha'$  came from the $i$th part of $\lambda$ (where parts of equal length are kept in the same relative order), then $j$ in $C$ is replaced by $i$ in $B$. Then $A \in \cA_\lambda$, $B \in \cO_{\lambda\beta}$, and setting 
$g^{-1}((\alpha, C)) = (A,B)$ gives $g(g^{-1}((\alpha, C))) = (\alpha, C)$ and $g^{-1}(g((A,B)) = (A,B)$. 
\end{proof}

\begin{example}\label{ex:OSP} 
Fix $\lambda=(3,2,2,1,1,1,1)$ and $\beta = (5,1,4,1)$. So $m_1 = 4$, $m_2 = 2$, and $m_3 = 1$.

Now consider 
$$A = ((\{1,2,4,7\},\{3,6\},\{5\}) \in \mathcal{A}_{\lambda} \quad \text{ and } \quad 
B = (\{1,3\},\{4\},\{2,5,7\},\{6\}) \in  \cO_{\lambda\beta}.$$
Then $\alpha' = \gamma(A) = (1,1,2,1,3,2,1)$,  corresponding to the rearrangement
$$\TikZ{
\foreach \x in {1,2,...,7}{\coordinate (t\x) at (\x,1); \coordinate (b\x) at (\x,0); }
\node[above]  at (1,1) {$\begin{matrix}\lambda_1 \\ 3\end{matrix}$};
\node[above] at (2,1) {$\begin{matrix}\lambda_2 \\ 2\end{matrix}$};
\node[above] at (3,1) {$\begin{matrix}\lambda_3 \\ 2\end{matrix}$};
\node[above] at (4,1) {$\begin{matrix}\lambda_4 \\ 1\end{matrix}$};
\node[above] at (5,1) {$\begin{matrix}\lambda_5 \\ 1\end{matrix}$};
\node[above] at (6,1) {$\begin{matrix}\lambda_6 \\ 1\end{matrix}$};
\node[above] at (7,1) {$\begin{matrix}\lambda_7 \\ 1\end{matrix}$};
\node[below] at (1,0) {$\begin{matrix}1\\ \alpha'_1\end{matrix}$};
\node[below] at (2,0) {$\begin{matrix}1\\ \alpha'_2\end{matrix}$};
\node[below] at (3,0) {$\begin{matrix}2\\ \alpha'_3\end{matrix}$};
\node[below] at (4,0) {$\begin{matrix}1\\ \alpha'_4\end{matrix}$};
\node[below] at (5,0) {$\begin{matrix}3\\ \alpha'_5\end{matrix}$};
\node[below] at (6,0) {$\begin{matrix}2\\ \alpha'_6\end{matrix}$};
\node[below] at (7,0) {$\begin{matrix}1\\ \alpha'_7\end{matrix}$};
\draw[thick,->] (t1) to (b5);
\draw[thick,->] (t2) to (b3);
\draw[thick,->] (t3) to (b6);
\draw[thick,->] (t4) to (b1);
\draw[thick,->] (t5) to (b2);
\draw[thick,->] (t6) to (b4);
\draw[thick,->] (t7) to (b7);
}, \quad \text{which induces} \quad \TikZ{
\foreach \x in {1,2,...,7}{\node (t\x) at (\x,2) {$\x$}; \node (b\x) at (\x,0)  {$\x$}; }
\draw[thick,->] (t1) to (b5);
\draw[thick,->] (t2) to (b3);
\draw[thick,->] (t3) to (b6);
\draw[thick,->] (t4) to (b1);
\draw[thick,->] (t5) to (b2);
\draw[thick,->] (t6) to (b4);
\draw[thick,->] (t7) to (b7);
}.
$$
The image of $B$ under this induced map is $C =  (\{5,6\}, \{1\}, \{2, 3, 7\}, \{4\})$. So $w_C = 5612374$, and the image of $\alpha'$ under the action of $w_C^{-1}$ on subscripts is 
$$w_C^{-1} : \ \TikZ{
\foreach \x in {1,2,...,7}{\coordinate (t\x) at (\x,1); \coordinate (b\x) at (\x,0); }
\node[above]  at (1,1) {$\begin{matrix}\alpha'_1 \\ 1\end{matrix}$};
\node[above] at (2,1) {$\begin{matrix}\alpha'_2 \\ 1\end{matrix}$};
\node[above] at (3,1) {$\begin{matrix}\alpha'_3 \\ 2\end{matrix}$};
\node[above] at (4,1) {$\begin{matrix}\alpha'_4 \\ 1\end{matrix}$};
\node[above] at (5,1) {$\begin{matrix}\alpha'_5 \\ 3\end{matrix}$};
\node[above] at (6,1) {$\begin{matrix}\alpha'_6 \\ 2\end{matrix}$};
\node[above] at (7,1) {$\begin{matrix}\alpha'_7 \\ 1\end{matrix}$};
\node[below] at (1,0) {$\begin{matrix}3\\ \alpha_1\end{matrix}$};
\node[below] at (2,0) {$\begin{matrix}2\\ \alpha_2\end{matrix}$};
\node[below] at (3,0) {$\begin{matrix}1\\ \alpha_3\end{matrix}$};
\node[below] at (4,0) {$\begin{matrix}1\\ \alpha_4\end{matrix}$};
\node[below] at (5,0) {$\begin{matrix}2\\ \alpha_5\end{matrix}$};
\node[below] at (6,0) {$\begin{matrix}1\\ \alpha_6\end{matrix}$};
\node[below] at (7,0) {$\begin{matrix}1\\ \alpha_7\end{matrix}$};
\draw[thick,->] (t5) to (b1);
\draw[thick,->] (t6) to (b2);
\draw[thick,->] (t1) to (b3);
\draw[thick,->] (t2) to (b4);
\draw[thick,->] (t3) to (b5);
\draw[thick,->] (t7) to (b6);
\draw[thick,->] (t4) to (b7);
}.$$
And, indeed, we see that $\alpha = (3,2,1,1,2,1,1) \cleq \beta$ and $C \in \OSP(\alpha, \beta)$.

We can see why it is necessary to record $\alpha$ as follows. For example, we consider inverting $g$ on the same $C$ as above, but now paired with $\alpha = (3,2,1,2,1,1,1)$.  Then the image of $\alpha$ under the action of $w_C$ on subscripts is 
$$w_C: \ \TikZ{
\foreach \x in {1,2,...,7}{\coordinate (t\x) at (\x,1); \coordinate (b\x) at (\x,0); }
\node[above]  at (1,1) {$\begin{matrix}\alpha_1 \\ 3\end{matrix}$};
\node[above] at (2,1) {$\begin{matrix}\alpha_2 \\ 2\end{matrix}$};
\node[above] at (3,1) {$\begin{matrix}\alpha_3 \\ 1\end{matrix}$};
\node[above] at (4,1) {$\begin{matrix}\alpha_4 \\ 2\end{matrix}$};
\node[above] at (5,1) {$\begin{matrix}\alpha_5 \\ 1\end{matrix}$};
\node[above] at (6,1) {$\begin{matrix}\alpha_6 \\ 1\end{matrix}$};
\node[above] at (7,1) {$\begin{matrix}\alpha_7 \\ 1\end{matrix}$};
\node[below] at (1,0) {$\begin{matrix}1\\ \alpha'_1\end{matrix}$};
\node[below] at (2,0) {$\begin{matrix}2\\ \alpha'_2\end{matrix}$};
\node[below] at (3,0) {$\begin{matrix}1\\ \alpha'_3\end{matrix}$};
\node[below] at (4,0) {$\begin{matrix}1\\ \alpha'_4\end{matrix}$};
\node[below] at (5,0) {$\begin{matrix}3\\ \alpha'_5\end{matrix}$};
\node[below] at (6,0) {$\begin{matrix}2\\ \alpha'_6\end{matrix}$};
\node[below] at (7,0) {$\begin{matrix}1\\ \alpha'_7\end{matrix}$};
\draw[thick,->] (t1) to (b5);
\draw[thick,->] (t2) to (b6);
\draw[thick,->] (t3) to (b1);
\draw[thick,->] (t4) to (b2);
\draw[thick,->] (t5) to (b3);
\draw[thick,->] (t6) to (b7);
\draw[thick,->] (t7) to (b4);
}.$$
So $A = (\{1,3,4,7\}, \{2,6\}, \{5\})$, which is different from the $A$ we started with above. The set $B$, though, is left unchanged.
\end{example}

\section{Relationships between bases}{\label{sec:btw}}
 
\subsection{The relationship between the type 1 and type 2 quasisymmetric power sums} 
 
To determine the relationship between the two different types of quasisymmetric power sums, we first use duality to expand the monomial quasisymmetric functions in terms of the type 2 quasisymmetric power sums.   Thus, from \eqref{eq:htopsi} and duality we obtain $$M_{\beta} = \sum_{\alpha \cgeq \beta} (-1)^{\ell(\beta)-\ell(\alpha)} \frac{\Pi_i \alpha_i}{\ell(\beta,\alpha)} \Phi_{\alpha}.$$

Then we expand the type 1 quasisymmetric power sums in terms of the monomial quasisymmetric functions \eqref{eq:PsiM} and apply substitution to obtain the following expansion of the type 1 quasisymmetric power sums into the type 2 quasisymmetric power sums:
 
 \begin{align*}
\Psi_{\alpha} &= \sum_{\beta \cgeq \alpha} \frac{z_{\alpha}}{\pi(\alpha, \beta)}M_{\beta} \\
 &= \sum_{\beta \cgeq \alpha} \frac{z_{\alpha}}{\pi(\alpha, \beta)} \sum_{\gamma\cgeq\beta} (-1)^{\ell(\beta)-\ell(\gamma)} \frac{\Pi_i \gamma_i}{\ell(\beta, \gamma)} \Phi_{\gamma} \\
 &= \sum_{\alpha \cleq \beta \cleq \gamma} (-1)^{\ell(\beta)-\ell(\gamma)} \frac{z_{\alpha} \Pi_i \gamma_i}{\pi(\alpha, \beta) \ell(\beta, \gamma)} \Phi_{\gamma}.
 \end{align*}
 
A similar process produces

\begin{align*}
\Phi_{\alpha} &= \sum_{\beta \cgeq \alpha} \frac{z_\alpha}{sp(\alpha,\beta)} M_{\beta} \\
&= \sum_{\beta \cgeq \alpha} \frac{z_\alpha}{sp(\alpha,\beta)} \sum_{\gamma\cgeq\beta} (-1)^{\ell(\beta)-\ell(\gamma)} lp(\beta, \gamma) \Psi_{\gamma} \\
&= \sum_{\alpha \cleq \beta \cleq \gamma} (-1)^{\ell(\beta)-\ell(\gamma)} \frac{z_\alpha lp(\beta, \gamma)}{sp(\alpha, \beta)} \Psi_{\gamma}.
\end{align*}
\subsection{The relationship between monomial and fundamental quasisymmetric functions}  Our next goal is to give the ``cleanest'' possible interpretation of the  expansions of the quasisymmetric power sums in the fundamental basis.  Towards this goal we first establish a more basic relationship between the $F$ basis and certain sums of monomials.

\begin{notn}[$\alpha^c$, $\alpha\wedge\beta$, $\alpha\vee\beta$]

 Given $\alpha\vDash n$, let $\alpha^c=\comp((\set(\alpha))^c)$.  Given a second composition $\beta$, $\alpha\wedge\beta$ denotes the finest (i.e.\ with the smallest parts) composition $\gamma$ such that $\gamma\cgeq\alpha$ and $\gamma\cgeq \beta$.  Similarly, $\alpha\vee \beta$ denotes the coarsest composition $\delta$ such that $\delta\cleq\alpha$ and $\delta \cleq \beta$.
 \end{notn} \begin{example} If $\alpha=(2,3,1)$ and $\beta=(1,2,2,1)$, then $\alpha^c=(1,2, 1,2), \alpha\wedge\beta=(5,1)$, and $\alpha\vee \beta=(1,1,1,2,1)$. 
  \end{example}
  
  \noindent The notation is motivated by the poset of sets ordered by containment (when combined with the bijection from sets to compositions). We note that $\set(\alpha\wedge\beta)=\set(\alpha) \cap \set(\beta)$ and 
$\set(\alpha \vee \beta)=\set(\alpha)\cup \set(\beta)$.

 We begin by writing the sum (over an interval in the refinement partial order) of quasisymmetric monomial functions as a sum of distinct fundamental quasisymmetric functions.      
\begin{lemma}\label{lem:moninterval}
Let $\alpha,\beta\vDash n$ with $\alpha\cleq\beta$.  Then 
\[\sum_{\delta: \alpha\cleq\delta\cleq\beta} M_\delta = \sum_{\beta\vee \alpha^c\cleq\delta\cleq\beta}(-1)^{\ell(\beta)-\ell(\delta)}F_\delta.\]
\end{lemma}

\begin{proof}
Let $\alpha, \beta\vDash n$ with $\alpha \cleq \beta$.  Then 
\begin{align}
\sum_{\alpha\cleq\delta\cleq\beta}M_\delta & =\sum_{\alpha\cleq\delta\cleq\beta}\sum_{\gamma\cleq\delta}(-1)^{\ell(\gamma)-\ell(\delta)}F_\gamma\nonumber\\
&=\sum_{\gamma\cleq\beta}(-1)^{\ell(\gamma)}F_\gamma\left(\sum_{\alpha\wedge\gamma\cleq\delta\cleq\beta}(-1)^{\ell(\delta)}\right).\label{eq:fexpansion}
\end{align}

Recall (see \cite{Sta99v1}) the M\"obius function for the lattice of subsets of size $n-1$, ordered by set inclusion.  If $S$ and $T$ are subsets of an $n-1$ element set with $T \subseteq S$, then $\mu(T,S)=(-1)^{|S-T|}$ and $(-1)^{|S-T|}= - \sum_{T\subseteq U\subset S} \mu(T,U)$. Thus, since compositions of $n$ are in bijection with subsets of $[n-1]$ and $\ell(\delta)=|\set(\delta)|+1$, when $\alpha\wedge\gamma \neq \beta$, we can write 
\begin{align*}
\sum_{\alpha\wedge\gamma\cleq\delta\cleq\beta} (-1)^{\ell(\delta)} &= (-1)^{\ell(\alpha\wedge\gamma)}+(-1)^{\ell(\beta)}\sum_{\alpha\wedge\gamma\prec\delta\cleq\beta}(-1)^{\ell(\delta)-\ell(\beta)} \\
&=(-1)^{\ell(\alpha\wedge\gamma)}+(-1)^{\ell(\beta)}\sum_{\alpha\wedge\gamma\prec\delta\cleq\beta}\mu(\set(\beta),\set(\delta)) \\
&=(-1)^{\ell(\alpha\wedge\gamma)}+(-1)^{\ell(\beta)+1}\mu(\set(\beta), \set(\alpha\wedge\gamma)) \\
&=(-1)^{\ell(\alpha\wedge\gamma)}+(-1)^{\ell(\beta)+1}(-1)^{\ell(\alpha\wedge\gamma)-\ell(\beta)}\\
&=0.\end{align*}

We can now rewrite \eqref{eq:fexpansion} as
\[\sum_{\gamma\cleq\beta}(-1)^{\ell(\gamma)}F_\gamma\left(\sum_{\alpha\wedge\gamma\cleq\delta\cleq\beta}(-1)^{\ell(\delta)}\right)=\sum_{\substack{\gamma\cleq\beta\\\beta=\gamma\wedge\alpha}} (-1)^{\ell(\gamma)+\ell(\alpha\wedge\gamma)}F_\gamma=\sum_{ \alpha^c\vee\beta \cleq \gamma\cleq \beta}(-1)^{\ell(\gamma)-\ell(\beta)}F_\gamma.\qedhere\]
\end{proof}

\subsection{The relationship between type 1 quasisymmetric power sums and fundamental quasisymmetric functions}

Recall Notation \ref{notn:hat} for the following.
\begin{thm}\label{thm:psitoF}
Let $\alpha\vDash n$.  Then 
\[\Psi_\alpha = \frac{z_\alpha}{n!}\sum_{\gamma\cgeq\alpha}|\{\sigma\in\Cons_\alpha: \widehat{\alpha(\sigma)}=\gamma\}|\sum_{\eta\cgeq \alpha^c}(-1)^{\ell(\eta)-1}F_{\gamma\vee\eta}.\]
\end{thm}

\begin{proof}
Let $\alpha\vDash n$.  We use $\one_{\mathcal{R}}$ to denote the characteristic function of the relation $\mathcal{R}$.

Combining the quasisymmetric monomial expansion of $\Psi_\alpha$ given in \eqref{eq:PsiM}, Lemma~\ref{lem:cons}, and Lemma~\ref{lem:moninterval}, we have
\begin{align*}
\Psi_\alpha &=\frac{z_\alpha}{n!}\sum_{\alpha\cleq\beta}|\Cons_{\alpha\cleq\beta}|M_\beta\\
&=\frac{z_\alpha}{n!}\sum_{\sigma\in \Cons_\alpha}\sum_{\alpha\cleq\delta\cleq \widehat{\alpha(\sigma)}}M_\delta \\
&=\frac{z_\alpha}{n!}\sum_{\sigma\in\Cons_\alpha}\sum_{\alpha^c\vee \widehat{\alpha(\sigma)}\cleq\delta\cleq\widehat{\alpha(\sigma)}}(-1)^{\ell(\widehat{\alpha(\sigma)})-\ell(\delta)}F_\delta \textrm{ (by Lemma~\ref{lem:moninterval})} \\
&=\frac{z_\alpha}{n!}\sum_{\delta\vDash n}(-1)^{\ell(\delta)}F_\delta \sum_{\sigma\in\Cons_\alpha}(-1)^{\ell(\widehat{\alpha(\sigma)})}\one_{\alpha^c\vee\widehat{\alpha(\sigma)}\cleq \delta \cleq \widehat{\alpha(\sigma)}}\\
&=\frac{z_\alpha}{n!}\sum_{\gamma\cgeq\alpha}|\{\sigma\in\Cons_\alpha:\widehat{\alpha(\sigma)}=\gamma\}|\sum_{\delta \vDash n}(-1)^{\ell(\gamma)-\ell(\delta)}F_\delta \one_{\alpha^c\vee\gamma\cleq \delta\cleq\gamma},  
\end{align*}
with the last equality holding since the  compositions $\widehat{\alpha(\sigma)}$ are coarsening of $\alpha$.  It is straightforward to check that given $\gamma \cgeq \alpha$ and $\delta \vDash n$, there exists $\eta \cgeq \alpha^c$ such that $\delta=\gamma\vee\eta$ if and only if $\delta \cgeq \alpha^c\vee\gamma$.  Then 
\begin{align*}
\Psi_\alpha&=\frac{z_\alpha}{n!}\sum_{\gamma\cgeq\alpha}|\{\sigma\in \Cons_\alpha: \widehat{\alpha(\sigma)}=\gamma\}|(-1)^{\ell(\gamma)}\sum_{\eta\cgeq\alpha^c}(-1)^{\ell(\gamma\vee\eta)}F_{\gamma\vee\eta}\\
&= \frac{z_\alpha}{n!}\sum_{\gamma\cgeq\alpha}|\{\sigma\in\Cons_\alpha: \widehat{\alpha(\sigma)}=\gamma\}|\sum_{\eta\cgeq \alpha^c}(-1)^{\ell(\eta)-1}F_{\gamma\vee\eta}.
\end{align*}
The final equality is established by noting that $\set(\gamma)\cap\set(\eta)=\emptyset$, so $\ell(\gamma\vee\eta)=|\set(\gamma\vee\eta)|+1=|\set(\gamma)|+|\set(\eta)|+1=\ell(\gamma)+\ell(\eta)-1$. 
\end{proof}
\begin{note}The $F_{\gamma \vee \eta}$'s are distinct in this sum, meaning the coefficient of $F_\delta$ is either 0 or is $$|\{ \sigma ~|~ \widehat{\alpha(\sigma)} = \gamma \}| (-1)^{\ell(\eta)-1}$$ when $\delta = \gamma \vee \eta$ for $\gamma \cgeq \alpha$ and $\alpha^c \cleq \eta$.  This follows from the fact that we can recover $\gamma$ and $\eta$ from $\gamma \vee \eta$ and $\alpha$, with $$\gamma=\comp(\set(\gamma \vee \eta)\cap \set(\alpha)),$$  $$\eta=\comp(\set(\gamma \vee \eta)\cap \set(\alpha)^c).$$  %This follows from the fact that the intervals $\alpha^c \vee \widehat{\alpha(\sigma)} \cleq \delta\cleq \widehat{\alpha(\sigma)}$ are disjoint: thinking just about what $\widehat{\alpha(\sigma)}$ can be, these are all possible coarsenings of $\alpha$.  Once that coarsening has been fixed, the viable $\delta$'s are any that, on the partitions from $\alpha$ do exactly the same thing on that coarsening, and then do whatever they want on the bars in $\alpha^c$. On those, the sign is computed from the coarsening of $\alpha$ and $\alpha$ itself; and the rest is computed as we wrote it above.  
\end{note}

\subsection{The relationship between type 2 quasisymmetric power sums and fundamental quasisymmetric functions}
The expansion of $\Phi_\alpha$ into fundamental quasisymmetric functions is somewhat more straightforward.  %Let $\OSP(\alpha)$ denote the set of all ordered set partitions of $\{1,\ldots, \ell(\alpha)\}$ and 
Let $m_i$ denote the number of parts of $\alpha \vDash n$ that have size $i$.  

\begin{thm}\label{thm:phitoF}
Let $\alpha\vDash n$.  Then 
\[\Phi_\alpha=\binom{m_1+\cdots+m_n}{m_1,\ldots,m_n}^{-1}\sum_{\gamma\vDash n} \left(\sum_{\beta\cgeq(\gamma\wedge \alpha)} (-1)^{\ell(\gamma)-\ell(\beta)}|\OSP(\alpha,\beta)|\right) F_{\gamma}.\]
\end{thm}

\begin{proof}
Let $\alpha\vDash n$.  Combining the quasisymmetric monomial expansion of $\Phi_\alpha$ and the fundamental expansion of $M_\beta$, gives 
\begin{align}
\Phi_\alpha&=\binom{m_1+\cdots+m_n}{m_1,\ldots,m_n}^{-1}\sum_{\beta\cgeq\alpha} |\OSP(\alpha,\beta)|M_\beta\\
&=\binom{m_1+\cdots+m_n}{m_1,\ldots,m_n}^{-1}\sum_{\beta\cgeq\alpha} |\OSP(\alpha,\beta)|\sum_{\beta\cgeq\gamma}(-1)^{\ell(\gamma)-\ell(\beta)}F_\gamma\\
&=\binom{m_1+\cdots+m_n}{m_1,\ldots,m_n}^{-1}\sum_{\gamma\vDash n}F_\gamma \left(\sum_{\beta\cgeq(\gamma\wedge \alpha)} (-1)^{\ell(\gamma)-\ell(\beta)}|\OSP(\alpha,\beta)|\right).\qedhere
\end{align}
\end{proof}
\subsection{The antipode map on quasisymmetric power sums}\label{sec:omega}
\begin{defn}[transpose, $\alpha^r,\alpha^t$] Let $\alpha^r$ give the reverse of $\alpha$. Then we call $\alpha^t=(\alpha^c)^r$ the  transpose of the composition $\alpha$. 
\end{defn}
The antipode map $S: \NS\rightarrow\NS$ on the Hopf algebra of quasisymmetric functions is commonly defined by $S(F_\alpha)=(-1)^{|\alpha|}F_{\alpha^t}$.  On the noncommutative functions, it is commonly defined as the automorphism such that $S(\e_n)=(-1)^n\h_n$.  It can equivalently be defined by $S(\r_\alpha)=(-1)^{|\alpha|}\r_{\alpha^t}$. Thus, for $f$ in $\QS$ and $g\in \NS$,  $$\<f,g\>=\<S(f),S(g)\>.$$  It can be easier to compute $S$ on a multiplicative basis, such as $\p$ or $\ptwo$ and then use duality to establish the result on the quasisymmetric side.  %We expand some of the details from \cite{GKLLRT94} does essentially this for $\p$; we rewrite the result for the quasisymmetric side.  

We start with the expansion of the $\p$ and $\ptwo$ in terms of the $\e$ basis in \cite[\S 4.5]{GKLLRT94}:
\begin{equation}\label{eq:powerine}
\p_n = \sum_{\alpha\vDash n} (-1)^{n-\ell(\alpha)}\alpha_1 \e_\alpha.
\end{equation}
It follows from \eqref{eq:powerinh} and \eqref{eq:powerine} that $S(\p_n)=-\p_n$.  Then 
\begin{align*}
S(\p_{\alpha})&=S(\p_{\alpha_1}\p_{\alpha_2}\cdots \p_{\alpha_k})\\
&=S(\p_{\alpha_k})S(\p_{\alpha_{k-1}})\cdots S(\p_{\alpha_1})\\
&=\left(-\p_{\alpha_k}\right)\cdots \left(-\p_{\alpha_1}\right)\\
&=(-1)^{\ell(\alpha)}\p_{\alpha^r}.
\end{align*}
 This result also follows from the fact that $\Psi$ is a primitive element.

\begin{thm}\label{thm:omega}
For $\alpha \vDash n$, $S(\Psi_\alpha) = (-1)^{\ell(\alpha)}\Psi_{\alpha^r}.$
\end{thm}
\begin{proof}
Let $\alpha \vDash n$.  Then
\[z_\alpha \delta_{\alpha, \beta}=\<\Psi_\alpha, \p_\beta\>=\<S(\Psi_\alpha), S(\p_\beta)\>=\<S(\Psi_\alpha),(-1)^{\ell(\beta)}\p_{\beta^r} \>=\<(-1)^{\ell(\beta)}S(\Psi_\alpha),\p_{\beta^r} \>,\]
so $S(\Psi_\alpha)=(-1)^{\ell(\alpha)}\Psi_{(\alpha)^r}$.
\end{proof}
Similarly, we have that $$S(\Phi_\alpha)=(-1)^{\ell(\alpha)}\Psi_{\alpha^r}.$$

There are considerable notational differences between various authors on the names of the well known automorphisms of $\QS$ and $\NS$, in part because there are two natural maps which descend to the well known automorphism $\omega$ in the symmetric functions.  Following both \cite{GKLLRT94} and \cite{LMvW}, we use $\omega(\e_n)=\h_n$ (where $\omega$ is an anti-automorphism) and $\omega(F_\alpha)=F_{\alpha^t}$ to define (one choice of) a natural analogue of the symmetric function case.  We can see, from the definition of $\omega$ and $S$ on the elementary symmetric functions, that the two maps vary by only a sign on homogeneous polynomials of a given degree.  In particular, 
$$\omega(\Psi_\alpha)=(-1)^{|\alpha|-\ell(\alpha)}\Psi_{\alpha^r},$$
$$\omega(\Phi_\alpha)=(-1)^{|\alpha|-\ell(\alpha)}\Phi_{\alpha^r}.$$
\section{Products of quasisymmetric power sums}\label{sec:products}

In contrast to the symmetric power sums, the quasisymmetric power sums are not a multiplicative basis.  This is immediately evident from the fact that $\Psi_{(n)}=p_{(n)}=\Phi_{(n)}$ but the quasisymmetric power sum basis is not identical to the symmetric power sums.    Thus the product of two elements of either power sum basis is more complex in the quasisymmetric setting than the symmetric setting.

\subsection{Products of type 1 quasisymmetric power sums}{\label{sec:product}}

We can exploit the duality of comultiplication in $\NS$ and multiplication in $\QS$. 
\begin{defn}[shuffle, $\shuffle$]
Let $[a_1,\cdots ,a_n]\shuffle[b_1,\cdots,b_m]$ give the set of shuffles of $[a_1,\cdots ,a_n]$ and $[b_1,\cdots,b_n]$; that is the set of all length $m+n$ words without repetition on  $
\{a_1,\cdots ,a_n\}\cup \{b_1,\cdots,b_n\}$ such that for all $i$, $a_i$ occurs before $a_{i+1}$ and $b_i$ occurs before $b_{i+1}$.
\end{defn}

Comultiplication for the noncommutative symmetric power sums (type 1) is given in~\cite{GKLLRT94} by \[\Delta ( \p_k) = 1 \oplus \p_k + \p_k \oplus 1.\]  Thus \[\Delta(\p_\alpha) = \prod_i \Delta (\p_{\alpha_i}) = \prod_i (1\oplus \p_{\alpha_i}+\p_{\alpha_i}\oplus 1) = \sum_{\substack{\gamma,\beta\\\alpha \in \gamma \shuffle \beta}} \p_{\gamma}\oplus \p_\beta.\]
\begin{notn}[$C(\alpha,\beta)$]  Let $a_j$ denote the number of parts of size $j$ in $\alpha$ and $b_j$ denote the number of parts of size $j$ in $\beta$.  Define $C(\alpha,\beta)=\prod_j \binom{a_j+b_j}{a_j}.$  A straightforward calculation shows that $C(\alpha,\beta) = z_{\alpha\cdot\beta}/(z_\alpha z_\beta)$.
\end{notn}

\begin{thm}\label{thm:productpower}
Let $\alpha$ and $\beta$ be compositions.  Then 
\[\Psi_\alpha \Psi_\beta = \frac{1}{C(\alpha,\beta)}\sum_{\gamma \in \alpha\shuffle\beta} \Psi_\gamma.\] 
\end{thm}

\begin{proof}
Let $\alpha$ and $\beta$ be compositions.  Then 
\begin{align}
\Psi_\alpha \Psi_\beta & = (z_\alpha \psi_\alpha)(z_\beta \psi_\beta)\nonumber\\
&= (z_\alpha z_\beta) (\psi_\alpha\psi_\beta). \label{eq:prod}
\end{align}
Since the $\psi$ are dual to the $\p$, we have that $\displaystyle{\psi_\alpha\psi_\beta=\sum_{\gamma \in \alpha\shuffle\beta}\psi_\gamma}$.  Note that for any rearrangement $\delta$ of $\gamma$, $z_\delta=z_\gamma$.  Thus, we can rewrite \eqref{eq:prod} as 
\[\Psi_\alpha\Psi_\beta = \frac{z_\alpha z_\beta}{z_{\alpha\cdot\beta}}\sum_{\gamma \in \alpha \shuffle \beta}\Psi_\gamma. \]
\end{proof}

In addition to this proof based on duality, we note that it is possible to prove this product rule directly using the monomial expansion of the quasisymmetric power sums.  We do this by showing that the coefficients in the quasisymmetric monomial function expansions of both sides of the product formula in Theorem~\ref{thm:productpower} are the same.
\begin{defn}[overlapping shuffle, $\cshuffle$] $\delta\cshuffle\beta$ is the set of {\em overlapping shuffles} of $\delta$ and $\eta$, that is, shuffles where a part of $\delta$ and a part of $\eta$ can be added to form a single part.

\end{defn}
\begin{lemma}\label{lem:productpi}
Let $\alpha \vDash m$, $\beta \vDash n$, and fix $\xi$ a coarsening of a shuffle of $\alpha$ and $\beta$.  Then
\[\binom{m+n}{m}\sum_{\substack{\delta\cgeq\alpha,\eta\cgeq \beta\\\text{s.t.\ }\xi\in\delta\cshuffle \eta}}\frac{m!}{\pi(\alpha,\delta)}\frac{n!}{\pi(\beta,\eta)} = \sum_{\substack{\gamma \in \alpha \shuffle \beta\\\gamma \cleq \xi}}\frac{(m+n)!}{\pi(\gamma,\xi)}.\]
\end{lemma}
\begin{proof} Let $\alpha \vDash m$, $\beta \vDash n$, and fix $\xi$, a coarsening of a shuffle of $\alpha$ and $\beta$.  Then $\xi=\xi_1,\ldots,\xi_k$ where each $\xi_i$ is a (known) sum of parts of $\alpha$ or $\beta$, or both.   %Let $M_{\delta,\eta}=\{\sigma \in S_m: \sigma \text{ is consistent with }\alpha\cleq\delta\}\times \{\tau \in S_n: \tau \text{ is consistent with }\beta\cleq \eta\}$.   Note that \[|M_{\delta,\eta}|=\frac{m!}{\pi(\alpha,\delta)}\frac{n!}{\pi(\beta,\eta)}.\]  
Let $Y_m= \{D\subseteq [m+n]: |D|=m\}$ and $B_{\xi,\alpha,\beta}=\{\gamma \in \alpha\shuffle\beta:\gamma\cleq\xi\}$.  We establish a bijection 
\[f: Y_m\times \bigcup_{\substack{\delta\cgeq\alpha,\eta\cgeq\beta\\\text{s.t.\ }\xi\in\delta\cshuffle\eta}}\left(\Cons_{\alpha\cleq\delta}\times\Cons_{\beta\cleq\eta}\right)\rightarrow \bigcup_{\gamma \in B_{\xi,\alpha,\beta}} (\Cons_{\gamma\cleq\xi}\times\{\gamma\}).\]

Let $(D, \sigma, \tau) \in Y_m\times \bigcup_{\substack{\delta\cgeq\alpha,\eta\cgeq\beta\\\text{s.t.\ }\xi\in\delta\cshuffle\eta}}\left(\Cons_{\alpha\cleq\delta}\times\Cons_{\beta\cleq\eta}\right)$.  Then $D=\{i_1<i_2<\ldots<i_m\}$.  To construct $(\word,\gamma)=f((D,\sigma,\tau))$: 
\begin{enumerate}
\item Create a word $\widetilde{\sigma}$ that is consistent with $\alpha\cleq\delta$ by replacing each $j$ in $\sigma$ with $i_j$ from $D$.  Similarly, use $[m+n]\setminus D$ to create $\widetilde{\tau}$ consistent with $\beta\cleq\eta$.
\item Arrange the parts of $\widetilde{\sigma}$ and $\widetilde{\tau}$ in a single permutation by placing the parts corresponding to $\alpha_i$ (resp. $\beta_i$) in the location they appear in $\xi$.  Finally, for all parts within a single part of $\xi$, arrange the sub-permutations so that the final elements of each sub-permutation creates an increasing sequence from left to right.  Note that this will keep parts of $\alpha$ in order since $\widetilde{\sigma}$ is consistent with $\alpha\cleq \delta$ and parts of $\alpha$ occurring in the same part of $\xi$ also occurred in the same part of $\delta$.  (An analogous statement is true for parts of $\beta$.)  
\item The resulting permutation is $\word=f((D,\sigma,\tau))$ and is an element of $\Cons_{\gamma\cleq\xi}$ where $\gamma$ is determined by the order of parts in $\word$ corresponding to $\alpha$ and $\beta$.
\end{enumerate}

Conversely, given $(\word',\gamma) \in \bigcup_{\gamma \in B_{\xi,\alpha,\beta}}(\Cons_{\gamma\leq\xi}\times\{\gamma\})$, construct a triple $(D',\sigma',\tau')$ by: 
\begin{enumerate}
\item In $\word'$, the $i$th block corresponds to the $i$th part of $\gamma$.  Place the labels in the $i$th block of $\word'$ in $D'$ if the $i$th part of $\gamma$ is from $\alpha$.
\item Let $\widetilde{\sigma}'$ be the subword of $\word'$ consisting of blocks corresponding to parts of $\alpha$, retaining the double-lines to show which parts of $\alpha$ were in the same part of $\xi$ to indicate $\delta\cgeq \alpha$.  Rewrite as a permutation in $S_m$ by standardizing in the usual way: replace the $i^{th}$ smallest entry with $i$ for $1\leq i \leq m$.  The resulting permutation $\sigma'$ is consistent with $\alpha\cleq\delta$.
\item Similarly construct $\tau'$ from the subword $\widetilde{\tau}'$ of $\word'$ consisting of parts corresponding to parts of $\beta$.  \qedhere
\end{enumerate}\end{proof} 

\begin{example}  Let $\alpha = (2,1,1,2)$, $\beta = (\underline{2},\underline{1})$ and $\xi=(2+1+\underline{2}, \underline{1}, 1+2)$.  
Then $(\delta,\eta)=((2+1,1+2), (\underline{2},\underline{1}))$.

Choose $D=\{1,2,5,6,7,9\}$, $\sigma = |\!|34|6|\!|2|15|\!| $, and $\tau=|\!|13|\!|2|\!|$.  Then $\widetilde{\sigma} = |\!|5 6|9|\!|2|17|\!|$ and $\widetilde{\tau}=|\!|3 8|\!|4|\!|$.  Then $\word = |\!|56|38|9|\!|4|\!|2|17|\!|$ and the corresponding shuffle $\gamma=(2,\underline{2},1,\underline{1},1,2)$.

Now, consider $\gamma'=(\underline{2},2,1,\underline{1},1,2)$ and $\word'=|\!|24|16|8|\!|9|\!|5|37|\!|.$  Then $\widetilde{\sigma}'=|\!|16|8|\!|5|37|\!|$ and $\widetilde{\tau}'=|\!|24|\!|9|\!|$.  Then $D'=\{1,3,5,6,7,8\}$, $\sigma'=|\!|14|6|\!|3|25|\!|$, and $\tau'=|\!|12|\!|3|\!|$.
\end{example}

We now use Lemma~\ref{lem:productpi} to offer a more combinatorial proof of Theorem~\ref{thm:productpower}.

\begin{proof}(of Theorem~\ref{thm:productpower})
Let $\alpha \vDash m$ and $\beta \vDash n$.  Then 
\begin{align}
\Psi_\alpha \Psi_\beta & = \left(\sum_{\delta\cgeq \alpha}\frac{z_\alpha}{\pi(\alpha,\delta)}M_\delta\right)\left(\sum_{\eta\cgeq \beta} \frac{z_\beta}{\pi(\beta,\eta)}M_\eta\right)\nonumber\\
&=z_\alpha z_\beta \sum_{\delta\cgeq\alpha}\sum_{\eta\cgeq \beta}\frac{1}{\pi(\alpha,\delta)\pi(\beta,\eta)}M_\delta M_\eta\nonumber\\
&=\frac{z_{\alpha\cdot\beta}}{C(\alpha,\beta)}\sum_{\delta\cgeq\alpha}\sum_{\eta\cgeq \beta}\frac{1}{\pi(\alpha,\delta)\pi(\beta,\eta)}\sum_{\zeta \in \delta \cshuffle\eta}M_\zeta\nonumber\\
&= \frac{z_{\alpha\cdot\beta}}{C(\alpha,\beta)}\sum_{\zeta \vDash m+n}M_\zeta \left(\sum_{\substack{(\delta,\eta): \delta \cgeq \alpha, \eta \cgeq \beta\\\zeta \in \delta\cshuffle\eta}}\frac{1}{\pi(\alpha,\delta)\pi(\beta,\eta)}\right).\label{eq:monex}
\end{align}

By Lemma~\ref{lem:productpi} we can rewrite \eqref{eq:monex} as
\begin{align*}
\Psi_\alpha\Psi_\beta &=\frac{z_{\alpha\cdot\beta}}{C(\alpha,\beta)}\sum_{\zeta \vDash m+n} M_\zeta \sum_{\substack{\gamma\in\alpha\shuffle\beta\\\gamma\cleq\zeta}}\frac{1}{\pi(\gamma,\zeta)}\\
&=\frac{1}{C(\alpha,\beta)}\sum_{\gamma \in \alpha \shuffle\beta}\sum_{\zeta \cgeq \gamma} \frac{z_{\gamma}}{\pi(\gamma,\zeta)}M_\zeta\\
&=\frac{1}{C(\alpha,\beta)}\sum_{\gamma \in \alpha\shuffle\beta}\Psi_\gamma.\qedhere
\end{align*}

\end{proof}

Now that we have a product formula for the quasisymmetric power functions, a more straightforward proof can be given for Theorem~\ref{thm:refine}. 
\begin{proof}(of Theorem~\ref{thm:refine})  We proceed by induction on $\ell(\lambda)$, the length of $\lambda$.  If $\ell(\lambda)=1$, then $\lambda=(n)$ and $p_{(n)}=m_{(n)}=M_{(n)}=\Psi_{(n)}.$  (This is because $\psi_{(n)}=\frac{1}{\pi((n),(n))}M_{(n)}=\frac{1}{n}M_{(n)}$ and $\Psi_{(n)}=z_{(n)}\psi_{(n)}=n\psi_{(n)}$.)

Suppose the theorem holds for partitions of length $k$ and let $\mu$ be a partition with $\ell(\mu)=k+1$.   Suppose  $\mu_{k+1}=j$ and let $\lambda=(\mu_1, \mu_2, \ldots, \mu_k)$. Let $m_j$ be the number of parts of size $j$ in $\mu$. Then, using the induction hypothesis and Theorem~\ref{thm:productpower}, we have \begin{equation}\label{ind} p_\mu=p_\lambda p_{(j)}= \left( \sum_{\substack{\alpha \vDash |\lambda|\\ \tilde{\alpha}=\lambda}} \Psi_\alpha\right)\Psi_{(j)}=  \sum_{\substack{\alpha \vDash|\lambda|\\ \tilde{\alpha}=\lambda}} \left( \Psi_\alpha\Psi_{(j)}\right)=\frac{1}{m_j} \sum_{\substack{\alpha \vDash |\lambda|\\ \tilde{\alpha}=\lambda}} \sum_{\gamma \in \alpha\shuffle (j)}  \Psi_\gamma.\end{equation} Here, we used the fact that, if $\tilde{\alpha}=\lambda$, then $\displaystyle C(\alpha, (j))=\binom{m_j}{m_j-1}=m_j$.

Suppose $\gamma\in \alpha \shuffle (j)$ for some  $\alpha \vDash |\lambda|$ such that $\tilde{\alpha}=\lambda$. Then $\gamma \vDash |\mu|$ and $\tilde{\gamma}=\mu$. Moreover, every composition $\theta\vDash |\mu|$ with $\tilde{\theta}=\mu$ belongs to  $ \alpha \shuffle (j)$ for some $\alpha\vDash |\lambda|$ with $\tilde{\alpha}=\lambda$. 

We write  $\gamma\vDash |\mu|$ with $\tilde{\gamma}=\mu$ as $\gamma^{(1)},J^{(1)},\gamma^{(2)},J^{(2)},\ldots, \gamma^{(q)},J^{(q)}$ where each $\gamma^{(i)}$ has no part equal to $j$ and each $J^{(i)}$ consists of exactly $r_i$ parts equal to $j$. We refer to $J^{(i)}$ as the $i$th block of parts equal to $j$. Here $r_i>0$ for $i=1, 2, \ldots, q-1$ and $r_q\geq 0$. Moreover, $r_1+r_2+\cdots +r_q=m_j$. Denote by $\alpha(\gamma, i)$ the composition obtained from $\gamma$ be removing the first $j$ in $J^{(i)}$ (if it exists). Then,  the multiplicity of $\gamma$ in $\alpha(\gamma, i)\shuffle (j)$ equals $r_i$ (since $(j)$ can be shuffled in $r_i$ different positions in the  $i$th block of parts equal to $j$ of $\alpha(\gamma, i)$ to obtain $\gamma$.) Then, the multiplicity of $\gamma$ in $\displaystyle\cup_{\tilde{\alpha}=\lambda}\{\alpha \shuffle (j) \mid \alpha \vDash |\lambda|, \tilde{\alpha}=\lambda\}$ equals $m_j$ and 
\[p_\lambda=\sum_{\substack{\beta \vDash |\mu|\\ \tilde{\beta}=\mu}} \Psi_\beta.\qedhere\]

\end{proof}

\subsection{Products of type 2 quasisymmetric power sums}
As with the type one quasisymmetric power sums, since $\Delta\ptwo_k=1\oplus\ptwo_k+\ptwo_k\oplus 1$, the product rule is \begin{equation}\label{eqn:productpower2}
\Phi_\alpha\Phi_\beta = \frac{1}{C(\alpha,\beta)}\sum_{\gamma\in \alpha\shuffle\beta}\Phi_\gamma.\end{equation}

Again, we can give a combinatorial proof of the product rule.  The proof of \eqref{eqn:productpower2} is almost identical to the proof of Theorem \ref{thm:productpower}, so we omit the details.  A significant difference is the proof of the analog of Lemma \ref{lem:productpi}, so we give the analog here.
\begin{lemma}Let $\alpha \vDash m$, $\beta \vDash n$, and fix $\xi$ a coarsening of a shuffle of $\alpha$ and $\beta$.  Then
	\[\sum_{\substack{\delta\cgeq\alpha,\eta\cgeq \beta\\\text{s.t.\ }\xi\in\delta\cshuffle \eta}}\frac{1}{\spab(\alpha,\delta)}\frac{1}{\spab(\beta,\eta)} = \sum_{\substack{\gamma \in \alpha \shuffle \beta\\\gamma \cleq \xi}}\frac{1}{\spab(\gamma,\xi)}.\]
\end{lemma}
\begin{proof}  First, note that $\prod_i\alpha_i$ and $\prod_i\beta_i$ occur in the denominator of every term in the left hand side of our desired equation, but $$\prod_i\alpha_i\prod_i\beta_i=\prod_i \gamma_i \text{ and } \ell(\alpha)+\ell(\beta)=\ell(\gamma)$$  for any $\gamma$ occuring in the right hand sum.  Then multiplying by $\ell(\gamma)!\prod_i\alpha_i\prod_i\beta_i $ on the left and right, we need to show
	\begin{align*}{\ell(\alpha)+\ell(\beta) \choose \ell(\alpha)}\sum_{\substack{\delta\cgeq\alpha,\eta\cgeq \beta\\\text{s.t.\ }\xi\in\delta\cshuffle \eta}}\frac{\ell(\alpha)!}{\prod_j\ell(\alpha^{(j)})!}\frac{\ell(\beta)!}{\prod_j\ell(\beta^{(j)})!} = \sum_{\substack{\gamma \in \alpha \shuffle \beta\\\gamma \cleq \xi}}\frac{\ell(\gamma)!}{\prod_j\ell(\gamma^{(j)})!}.
	\end{align*}
	Equivalently, we need to show 
		\begin{align}\label{eq:power2eq15}{\ell(\alpha)+\ell(\beta) \choose \ell(\alpha)}\sum_{\substack{\delta\cgeq\alpha,\eta\cgeq \beta\\\text{s.t.\ }\xi\in\delta\cshuffle \eta}}|\OSP(\alpha,\delta)||\OSP(\beta,\eta)| = \sum_{\substack{\gamma \in \alpha \shuffle \beta\\\gamma \cleq \xi}}|\OSP(\gamma,\xi)|.
	\end{align}
	
	For a given choice of $\delta$ and $\eta$ satisfying the conditions on the left, select $S$ and $T$, ordered set partitions in $\OSP(\alpha,\delta)$ and $\OSP(\beta,\eta)$ respectively. Pick a subset $U$ of size $\ell(\alpha)$ from the first $\ell(\alpha)+\ell(\beta)$ positive integers and re-number the elements in $S$ and $T$ in order, such that the elements of $S$ are re-numbered with the elements of $U$ and the elements of $T$ are re-numbered with  elements of $U^c$ to form $\tilde{S}$ and $\tilde{T}$ respectively.  Going forward, consider each of the subsets as lists, with the elements listed in increasing order.   Use the subsets to assign an additional value to each part of $\alpha$ or $\beta$, working in order.  Say that $f(\alpha,i)=m$ if $\alpha_i$ occurs as the $k$th element in $\alpha^{(j)}$ and $m$ is the $k$th element in $\tilde{S}_j$.  Similarly say that $f(\beta,i)=m$ if $\beta_i$ occurs as the $k$th element in $\beta^{(j)}$ and $m$ is the $k$th element in $\tilde{T}_j$.   Note that each choice of $\delta$ and $\eta$ gives a refinement of $\xi$, with each $\xi_i$ a sum of parts of $\alpha$ and $\beta$.  Then sort the parts of $\alpha$ and $\beta$ to create $\gamma$, such that parts of $\alpha$ occur in order and parts of $\beta$ occur in order, and the following additional rules are satisfied:  Let $\alpha_i$ be one of the parts that forms $\delta_j$ which in turn is used to form $\xi_k$ and $\beta_l$ be one of the parts that forms $\eta_m$ which in turn is used form $\xi_n$.  Then
 \begin{itemize}
	 	\item if $k>n$ (i.e.\ $\beta_l$ is an earlier subpart of $\xi$ than $\alpha_i$ is), $\alpha_i$ occurs after $\beta_l$,
	 	\item if $k<n$,(i.e.\ $\alpha_i$ is an earlier subpart of $\xi$ than $\beta_l$ is) $\alpha_i$ occurs before $\beta_l$,
	 	\item if $k=n$ (i.e. $\alpha_i$ and $\beta_l$ eventually make up the same part of $\xi$), $\alpha_i$ is left of $\beta_l$ iff $f(\alpha,i)<f(\beta,l)$. 
	 \end{itemize}
	 Finally, create an element of $\OSP(\gamma,\xi)$ by placing $p$ in the $q$th part if   $f(\alpha,i)=p$ (or $f(\beta,l)=p$) and $\alpha_i$ ($\beta_l$ respectively) is one of the parts used to form $\gamma^{(q)}$.  Note that this map is bijective; since the the parts of $\alpha$ and $\beta$ which occur in the same part of $\gamma$ are sorted by the value they are assigned in the set partition, we can recover from the set partition which integers were assigned to each part (and what $U$ was, by looking at which numbers are assigned to parts corresponding to $\alpha$).
     \end{proof}
     \begin{example}Let $\alpha=(1,2,1)$ and $\beta=(\underline{1},\underline{1},\underline{2})$.  Let $\delta=(1+2,1)$  $\eta=(\underline{1},\underline{1}+\underline{2})$, and $\xi=(1+2+\underline{1},\underline{1}+\underline{2},1)$.  ($\xi$ here is fixed before $\delta$ and $\eta$, but how we write it as an overlapping shuffle depends on their choice.)  Finally let $S=(\{1,3\},\{2\})$, $T=(\{3\},\{1,2\})$, and $U=\{1,4,6\}$.  Then $\tilde{S}=(\{1,6\},\{4\})$ and $\tilde{T}=(\{5\},\{2,3\})$.  Next, we reorder the second and third subparts of $\xi$ to get $\gamma=(1,\overline{1},2,\overline{1},\overline{2},1)$ and the final ordered set partition is $(\{1,5,6\},\{2,3\},\{4\})$. 
     \end{example}
\subsection{The shuffle algebra}

Let $V$ be a vector space with basis $\{v_a \}_{a \in \mathfrak{U}}$ where $\mathfrak{U}$ is a totally ordered set.  For our purposes, $\mathfrak{U}$ will be the positive integers.  The \emph{shuffle algebra} $sh(V)$ is the Hopf algebra of the set of all words with letters in $\mathfrak{U}$, where the product is given by the shuffle product $v \shuffle w$ defined above.  The shuffle algebra and $QSym$ are isomorphic as graded Hopf algebras~\cite{GriRei14}.  We now describe a method for generating $QSym$ through products of the type 1 quasisymmetric power sums indexed by objects called \emph{Lyndon words}; to do this we first need several definitions.

A \emph{proper suffix} of a word $w$ is a word $v$ such that $w=uv$ for some nonempty word $u$.  The following total ordering on words with letters in $\mathfrak{U}$ is used to define Lyndon words.  We say that $u \le_L v$ if either

\begin{itemize}
\item $u$ is a prefix of $v$,   \; \; \; or
\item $j$ is the smallest positive integer such that $u_j \not= v_j$ and $u_j<v_j$.
\end{itemize}

Otherwise $v \le_L u$.  If $w=w_1 w_2 \cdots w_k$ is a nonempty word with $w_i \in \mathfrak{U}$ for all $i$, we say that $w$ is a \emph{Lyndon word} if every nonempty proper suffix $v$ of $w$ satisfies $w < v$.  Let $\mathcal{L}$ be the set of all Lyndon words.  Radford's Theorem~\cite{Rad79}, (Theorem 3.1.1 (e)) states that if $\{ b_a \}_{a \in \mathfrak{U}}$ is a basis for a vector space $V$, then $\{b_w\}_{w \in \mathcal{L}}$ is an algebraically independent generating set for the shuffle algebra $Sh(V)$.  To construct a generating set for $Sh(V)$, first define the following operation (which we will call an \emph{index-shuffle}) on basis elements $b_{\alpha}$ and $b_{\beta}$: 
$$b_{\alpha} \underline{\shuffle} b_{\beta} = \sum_{\gamma \in \alpha \shuffle \beta} b_{\gamma}.$$  Recall that $$\Psi_\alpha \Psi_\beta = \frac{1}{C(\alpha,\beta)}\sum_{\gamma \in \alpha\shuffle\beta} \Psi_\gamma,$$ where $C(\alpha,\beta) = z_{\alpha\cdot\beta}/(z_\alpha z_\beta)$.  Then $\Psi_\alpha \underline{\shuffle} \Psi_\beta = C(\alpha, \beta) \Psi_{\alpha} \Psi_{\beta}$.  Since Radford's theorem implies that every basis element $b_{\alpha}$ can be written as a linear combination of index shuffles of basis elements indexed by Lyndon words, every basis element $\Psi_{\alpha}$ can be written as a linear combination of products of type 1 quasisymmetric power sums indexed by Lyndon words and we have the following result.

\begin{thm}
The set $\{ \Psi_C | C \in \mathcal{L} \}$ freely generates $QSym$ as a commutative $\mathbb{Q}$-algebra.
\end{thm}

\begin{example}
Since $231$ can be written as $23 \shuffle 1 - 2 \shuffle 13 + 132$, $$\Psi_{231} = C(23,1)\Psi_{23} \Psi_1 - C(2,13) \Psi_2 \Psi_{13} + \Psi_{132} = \Psi_{23} \Psi_1 - \Psi_2 \Psi_{13} + \Psi_{132}. $$
\end{example}

\section{Plethysm on the quasisymmetric power sums}{\label{sec:plethysm}}  The symmetric power sum basis $\{p_\lambda\}_{\lambda\vdash n}$ plays a particularly important role in the language of $\Lambda$-rings.  It is natural to hope that one of the quasisymmetric power sums might play the same role here, and it was this motivation that initially piqued the authors' interest in the quasisymmetric power sums.  This seems not to be the case, so one might take the next section as a warning to similarly minded individuals that this does not appear to be a productive path for studying quasisymmetric plethysm.  To explain the differences between this and the symmetric function case, we remind the reader of the symmetric function case first.
\subsection{Plethysm and symmetric power sums}
Recall that plethysm is a natural (indeed even universal in some well defined functorial sense) $\Lambda$-ring on the symmetric functions.  Following the language of \cite{knutson2006lambda}, recall that a pre-$\Lambda$-ring is a commutative ring $R$ with identity and a set of operations for $i\in \{0,1,2,\cdots\}$ $\lambda^i:R\rightarrow R$ such that for all $r_1,r_2\in R$:
\begin{itemize}
\item $\lambda^0(r_1)=1$
\item $\lambda^1(r_1)=r_1$
\item $\lambda^n(r_1+r_2)=\sum_{i=0}^n \lambda^i(r_1)\lambda^{n-i}(r_2)$
\end{itemize}
To define a $\Lambda$-ring, use $e^X_i$ and $e^Y_i$ as the elementary symmetric functions $e_i$ in the $X$ or $Y$ variables, and define the universal polynomials $P_n$ and $P_{m,n}$ by  $$\sum_{n=0}P_n(e^X_1,\cdots,e^X_n;e^Y_1,\cdots, e^Y_n)t^n=\prod_{i,j}(1-x_{i,j}t),$$ and $$\sum_{n=0}P_{n,m}(e^X_1,\cdots,e^X_{nm})t^n=\prod_{i_1<\cdots<i_m }(1-x_{i_1}x_{i_2}\cdots x_{i_m}t).$$ 
Then a pre-$\Lambda$-ring is by definition a $\Lambda$-ring if 
\begin{itemize}
\item for all $i>1$, $\lambda^i(1)=0$;
\item for all $r_1,r_2\in R$, $n\geq 0$, $$\lambda^n(r_1r_2)=P_n(\lambda^1 (r_1),\cdots, \lambda^n (r_1);\lambda^1 (r_2),\cdots, \lambda^n (r_2));$$
\item for all $r\in R$, $m,n\geq 0$, $$\lambda^m(\lambda^n (r))=P_{m,n}(\lambda^1 (r),\cdots, \lambda^{mn} (r)).$$
\end{itemize}
These operations force the $\lambda_i$ to behave like exterior powers of vector spaces (with sums and products of $\lambda_i$ corresponding to exterior powers of direct sums and tensor products of vector spaces), but are not always so helpful to work with directly.  In the classical case, one works more easily indirectly by defining a new series of operations called the Adams operators $\Psi^n:R\rightarrow R$ by the relationship (for all $r\in R$)
\begin{equation}\label{eq:adams}\frac{d}{dt}\log \sum_{i\geq 0}t^i\lambda^i(r)=\sum_{i=0}^\infty (-1)^i\Psi^{i+1}(r)t^i.\end{equation}
Note that while use $\Psi$ in this section for both the power sums and the Adams operators, the basis elements have subscripts and the Adams operators superscripts.  Standard literature uses $\Psi$ for the Adams operators, so this follows the usual convention.  Moreover, there is quite a close connection between the two, as mentioned below.
\begin{thm}[\cite{knutson2006lambda}]\label{thm:adams}
If $R$ is torsion free, $R$ is a $\Lambda$-ring if and only if for all $r_1,r_2\in R$,
\begin{enumerate}
\item \label{it:1}$\Psi^i(1)=1$,
\item\label{it:2} $\Psi^i(r_1r_2)=\Psi^i(r_1)\Psi^i(r_2)$, and 
\item \label{it:3} $\Psi^i(\Psi^j(r_1))=\Psi^{ij}(r_1)$.
\end{enumerate}
\end{thm}
Since (\ref{eq:adams}) defining the Adams operators is equivalent to (\ref{eq:pfrome}), this suggests that simple operations on the symmetric power sum functions should give a $\Lambda$-ring.  This is exactly the case for the ``free $\Lambda$-ring on one generator'', where we start with $\Lambda$ a polynomial ring in infinitely many variables $\mathbb{Z}[x_1, x_2,\cdots]$ and let (for any symmetric function $f\in\mathbb{Z}[x_1, x_2,\cdots] $) $$\lambda^i(f)=e_i[f].$$
The $[\cdot]$ on the right denotes plethysm.  One can make an equivalent statement using the Adams operators and the symmetric power sum $p_i$:
$$\Psi^i(f)=p_i[f].$$
This implies that for all $i\geq 0$, $f,g\in\mathbb{Z}[x_1, x_2,\cdots]$ (symmetric functions, although $f$ and $g$ can in fact be any integral sum of monomials)
\begin{enumerate}
\item$p_i[1]=1$,
\item $p_i[f+g]=p_i[f]+p_i[g]$ and 
\item $p_m[fg]=p_m[f]p_m[g]$.
\end{enumerate}  (Note that the first and third items follow directly from (\ref{it:1}) and (\ref{it:2})  in Theorem \ref{thm:adams}.  The second follows from the additive properties of a $\Lambda$-ring.)  This is the context in which plethysm is usually directly defined, such that for $f,g$ symmetric functions (and indeed with $f$ allowed much more generally to be a sum of monomials) one more generally calculates $g[f]$ by first expressing $g$ in terms of the power sums and then using the above rules, combined with the following (which allows one to define plethysm as a homomorphism on the power sums):
\begin{itemize}
\item For any constant $c\in \mathbb{Q}$ (or generalizing, for $c$ in an underlying field $K$), $c[f]=c$.
\item For $m\geq 1$, $g_1,g_2$ symmetric functions, $(g_1+g_2)[f]=g_1[f]+g_2[f]$.
\item For $m\geq 1$, $g_1,g_2$ symmetric functions, $(g_1g_2)[f]=g_1[f]g_2[f]$.
\end{itemize}
In this context, $$s_\lambda[s_\mu]=\sum_{\gamma}a_{\lambda,\mu}^\nu s_\nu$$ where the $a_{\lambda,\mu}^\nu$ correspond to the well-known Kronecker coefficients and $f[x_1+\cdots +x_n]$ is just $f(x_1,\cdots,x_n)$. See \cite{LeuRem2007Computational} for a fantastic exposition, starting from this point of view.

\subsection{Quasisymmetric and noncommutative symmetric plethysm}  While one can modify the definition slightly to allow evaluation of $g[f]$ when $f$ is not a symmetric function, the case is not so simple when $g$ is no longer symmetric.  Krob, Leclerc, and Thibon~\cite{KLT97Noncommutative} are the first to examine this case in detail; they begin by looking at $g$ in the noncommutative symmetric functions, but $g$ in the quasisymmetric functions can be defined from there by a natural duality.  (It is not so natural to try to find an analogue directly on the quasisymmetric function side, since one needs an analogue of the elementary symmetric functions to begin.)   A peculiarity of this case is that the order of the monomials in $f$ matters (as the evaluation of noncommutative symmetric functions or quasisymmetric functions depends on the order of the variables), and one can meaningfully define a different result depending on the choice of monomial ordering.   

As is suggested by the formalism of $\Lambda$-rings, Krob, Leclerc, and Thibon~\cite{KLT97Noncommutative} begin by essentially defining the natural analogue of a pre-$\Lambda$-ring on the  homogeneous noncommutative symmetric functions, and thus by extension on the elementary noncommutative symmetric functions  They do this in a way that guarantees the properties of a pre-$\Lambda$-ring as follows.
Use $A\oplus B$ for the addition of ordered noncommuting alphabets (with the convention that terms in $A$ preceed terms in $B$) and $H(X;t)$ for the generating function of the homogeneous noncommutative symmetric functions on the alphabet $X$.  Then define $H(A \oplus B; T)$ as follows.  $$H(A\oplus B;t)=\sum_{n\geq 0}\h_{n}[A\oplus B]t^n:= H(A;t)H(B;t).$$

Already, this is enough to show that plethysm on the noncommutative power sums (or by duality the quasisymmetric power sums) is not as nice. Using $\p(X;t)$ for the generating function of the noncommuting symmetric power sums, \cite{KLT97Noncommutative} show that
$$\p(A\oplus B;t)=H(B;t)^{-1}\p(A;t)H(B;t)+\p(B;t).$$  This is, of course, in contrast to simple easy to check symmetric function expansion  $$p[X+Y;t] =p[X;t]+p[Y;t],$$ for $p[X;t]$ the symmetric power sum generating function.
Moreover, they show that the case is equally complex for the type 2 case.
The takeaway from this computation is that one can define a $\Lambda$-ring or a pre-$\Lambda$-ring in the noncommutative symmetric functions and then define Adams operators by one of two relationships between the power sums and the elementary (or equivalently homogeneous) noncommutative symmetric functions, but the resulting relationships on the Adams operators, that is the analogue of  Theorem \ref{thm:adams}, can be far more complicated than working with plethysm directly using the elementary nonsymmetric functions or (by a dual definition) the monomials in the quasisymmetric functions.  A succinct resource to the latter is \cite{BKNMPT01overview}.  

In theory, one could work in reverse, defining an operation ``plethysm" on either the type 1 or the type 2 power sums and extending it as a homomorphism to the quasisymmetric or noncommutative symmetric functions.  In practice, besides the more complicated plethysm identities on the power sums, most of the natural relationships commonly used in (commuting) symmetric function calculations are naturally generalized by plethysm as defined in \cite{KLT97Noncommutative}.  (Here for example, the addition of alphabets corresponds to the coproduct, as in the symmetric function case.)  Therefore the perspective of \cite{KLT97Noncommutative} seems to result in a more natural analogue of plethysm than any choice of homomorphism on the quasisymmetric power sums.

\section{Future directions}
We suggest a couple of possible directions for future research.
\subsection{Murnaghan-Nakayama style rules}
The Murnaghan-Nakayama Rule provides a formula for the product of a power sum symmetric function indexed by a single positive integer and a Schur function expressed in terms of Schur functions.  This rule can be thought of as a combinatorial method for computing character values of the symmetric group.  Tewari~\cite{Tew16} extends this rule to the noncommutative symmetric functions by producing a rule for the product of a type 1 noncommutative power sum symmetric function and a noncommutative Schur function.  LoBue's formula for the product of a quasisymmetric Schur function and a power sum symmetric function indexed by a single positive integer~\cite{LoB15} can be thought of as a quasisymmetric analogue of the Murnaghan-Nakayama rule, although there are several alternative approaches worth exploring.

The quasisymmetric power sum analogues, unlike the power sum symmetric functions and the noncommutative power sum symmetric functions, are not multiplicative, meaning a rule for multiplying a quasisymmetric Schur function and a quasisymmetric power sum indexed by a single positive integer does not immediately result in a rule for the product of a quasisymmetric Schur function and an arbitrary quasisymmetric power sum.  It is therefore natural to seek a new rule for such a product.  

Also recall that there are several natural quasisymmetric analogues of the Schur functions; in addition to the quasisymmetric Schur functions, the fundamental quasisymmetric functions and the dual immaculate quasisymmetric functions can be thought of as Schur-like bases for $QSym$.  Therefore it is worth investigating a rule for the product of a quasisymmetric power sum and either a fundamental quasisymmetric function or a dual immaculate quasisymmetric function.

\subsection{Representation theoretic interpretations}  The symmetric power sums play an important role in connecting representation theory to symmetric function theory via the Frobenius characteristic map $\mathcal{F}$.   In particular, if $C_\lambda$ is a class function in the group algebra of $S_n$, one can define the Frobenius map by $\mathcal{F}(C_\lambda)=\frac{p_\lambda}{z_\lambda}$.  With this definition, one can show that $\mathcal{F}$ maps the irreducible representation of $S_n$ indexed by $\lambda$ to the schur function $s_\lambda$.

Krob and Thibon~\cite{krob1997noncommutative} define quasisymmetric and noncommutative symmetric characteristic maps; one takes irreducible representations of the $0$-Hecke algebra to the fundamental quasisymmetric basis, the other takes indecomposable representations of the same algebra to the ribbon basis.  It would be interesting to understand if these maps could be defined equivalently and usefully on the quasisymmetric and noncommutative symmetric power sums.
\begin{akn} 
We would like to express our gratitude to BIRS and the organizers of Algebraic Combinatorixx II for bringing the authors together and enabling the genesis of this project, and to the NSF for supporting further collaboration via the second author's grant (DMS-1162010).
\end{akn}

\bibliographystyle{amsplain}
\providecommand{\bysame}{\leavevmode\hbox to3em{\hrulefill}\thinspace}
\providecommand{\MR}{\relax\ifhmode\unskip\space\fi MR }
% \MRhref is called by the amsart/book/proc definition of \MR.
\providecommand{\MRhref}[2]{%
  \href{http://www.ams.org/mathscinet-getitem?mr=#1}{#2}
}
\providecommand{\href}[2]{#2}

\end{document}